\documentclass[11pt]{article}
\usepackage{amssymb,latexsym,amsmath,amsbsy,amsthm,amsxtra,amsgen}
\newfont{\msbm}{msbm10 at 11pt}
\newcommand {\R} {\mbox{\msbm R}}

\newcommand {\N} {\mbox{\msbm N}}

\def\be{\begin{equation}}
\def\ee{\end{equation}}
\def\ba{\begin{align}}
\def\ea{\end{align}}

\newtheorem{Theo}{Theorem}
\newtheorem{Lemma}[Theo]{Lemma}
\newtheorem{Cor}[Theo]{Corollary}
\newtheorem{Prop}[Theo]{Proposition}

\begin{document}
\title{A Hierarchical Probability Model of Colon Cancer}
\author{by Michael Kelly \\ University of California, San Diego}
\maketitle

\footnote{Contact: mbkelly@math.ucsd.edu, Mathematics Department, UC San Diego, 9500 Gilman Dr. \#0112, La Jolla, CA 92093, U.S.A.}


\begin{abstract}
We consider a model of fixed size $N = 2^l$ in which there are $l$ generations of daughter cells and a stem cell. In each generation $i$ there are $2^{i-1}$ daughter cells. At each integral time unit the cells split so that the stem cell splits into a stem cell and generation 1 daughter cell and the generation $i$ daughter cells become two cells of generation $i+1$. The last generation is removed from the population. The stem cell gets first and second mutations at rates $u_1$ and $u_2$ and the daughter cells get first and second mutations at rates $v_1$ and $v_2$. We find the distribution for the time it takes to get two mutations as $N$ goes to infinity and the mutation rates go to 0. We also find the distribution for the location of the mutations. Several outcomes are possible depending on how fast the rates go to 0. The model considered has been proposed by Komarova (2007) as a model for colon cancer.
\end{abstract}


\section{Introduction}
In the 1950's Armitage and Doll \cite{ArDo} proposed that cancer may be the end result of an accumulation of two or more cell mutations. Komarova \cite{Kom} discusses three mathematical models which may be used to model the mutations that lead to cancer. The first is the Moran model, which may be used to model cancers in liquids such as Leukemia. In this model there is a fixed population of size $N$. There is a rate $\mu$ at which cells are getting mutations. Each cell in the population dies at rate 1 and is replaced by any individual in the population, including itself, with equal probability. The second is a spatial model which may be used to model cancers in solid tissues. This model is similar to the Moran model except that the cells are given spatial locations and when they die they are only replaced by nearby cells. The third model, the one we focus on in this paper, is referred to as the hierarchical model in \cite{Kom}. This model may be used for colon cancer.

As discussed in \cite{Kom}, many cells in the human body, including those in the colon, go through a three step process. It begins with a stem cell which will stay in the population for a long time and have many descendants. Some of these descendants will also be stem cells, but others will be differentiated progenitor cells. The progenitor cells, or what we shall refer to as daughter cells in this paper, will split into more daughter cells. The number of times these cells split is dependent upon what organ of the body they are in. We will refer to the number of splits that a daughter cell has undergone as the generation of the cell. Once the cells split enough times they reach maturity and are swept out of the population in a biological process called apoptosis.

The colon is lined with crypts that contain pockets of cells. The cells in the colon, as described by Komarova in \cite{KWa}, are such that stem cells reside at the bottom of the crypt and the daughters migrate up the crypt so that the higher generation daughter cells are near the top. We assume that cancer is the result of two mutations, as is done in \cite{Kom}. There are three ways in which the mutations may occur. The stem cell may acquire both mutations so that cancer is a result of mutations of the stem cell only. It is possible that stem cell receives the first mutation and a daughter cell gets the second, or a daughter cell and one of its descendants will each receive mutations before they are swept from the crypt. In \cite{Kom} these cases are abbreviated as ss, sd and dd respectively.

The Hierarchical model shall be referred to as $H_1$. This model has a fixed population of size $N = 2^l$ where $l$ is the number of generations of daughter cells in the crypt. There is one stem cell and for $k \in \{1, 2, \dots, l\}$ there are $2^{k-1}$ daughter cells of generation $k$. We start with a full
crypt and no mutations. At each integral time unit all of the cells split in the following way:

\begin{itemize}
\item The stem cell splits into a stem cell and a generation 1 daughter cell.
\item For each generation $k$ with $1 \leq k \leq l-1$, a daughter cell of generation $k$ will split into two cells of generation $k+1$.
\item The daughter cells of generation $l$ undergo apoptosis and are swept from the population.
\end{itemize}
Notice that the generations are constant size throughout time. The cells will accumulate mutations via Poisson processes. A cell with 0, 1 or 2 mutations is called a type-0, type-1 or type-2 cell respectively. A mutation which occurs on a type-0 or type-1 cell is called a type-1 or type-2 mutation respectively. This terminology is used so that a mutation that makes a cell type-2 is called a type-2 mutation. Once a type-2 mutation occurs the colon is assumed to have cancer. The cells will each have two Poisson processes marking them, one which will cause type-1 mutations and one which will cause type-2 mutations. The first Poisson process that marks a cell will only cause a type-1 mutation if the cell is a type-0. If a mark of the Poisson process occurs while the cell is not a type-0 then the mutation is rejected. Likewise, the second Poisson process only causes mutations on type-1 cells. If a mark from this Poisson process occurs on a cell while it is type-1 then the cell becomes type-2, but if the cell is not a type-1 then nothing happens. All of the Poisson processes are independent. The mutations are passed to the descendants when a cell splits. It is sometimes convenient to think of the cells as fixed in a binary tree and the mutations as traveling through the tree in a direction which takes them from the root to the leaves. Because of this we will often refer to the sequence of stem cells as the stem cell and we fix the Poisson processes that are marking the cells on particular locations in the tree.

The rates at which the stem cell acquires type-1 and type-2 mutations are $u_1$ and $u_2$ respectively. The rates at which the daughter cells get type-1 and type-2 mutations are $v_1$ and $v_2$ respectively. Each of the rates are functions of $N$ and will approach 0 as $N$ approaches infinity. We will always consider what happens as $N$ goes to infinity. All limits will be assumed as taking $N$ to infinity unless otherwise stated.

We let $\tau(A_i)$ be the first time that any cell gets a type-2 mutation where $A_i$ refers to a model. We call a type-1 mutation to a cell which has a type-2 descendant successful. A type-1 mutation to a stem cell is always successful and a type-1 mutation to a daughter is successful if the daughter has a type-2 descendent before its progeny is washed from the population. We will call the successful type-1 mutation whose type-2 descendant is the first type-2 to occur the cancer causing type-1 mutation. Being the cancer causing type-1 mutation is not equivalent to being the first successful type-1 mutation. We also define random variables $\sigma(A_i)$ and $\rho(A_i)$ to be the depth of the colon at which the cancer causing type-1 and first type-2 mutations occur, respectively. More precisely, if the cancer causing type-1 mutation occurs in generation $j$ then we define $\sigma(A_i) = j/l$ and if the first type-2 mutation occurs in generation $k$ then we define $\rho(A_i) = k/l$. If the cancer causing type-1 mutation or first type-2 mutation occur on the stem cell then $\sigma(A_i) = 0$ or $\rho(A_i) = 0$ respectively.

The above establishes most of the notation that will be used throughout this paper, but some more will be included here. For any real number $a$ we define $a^+ = a \vee 0$. For functions $f(x)$ and $g(x)$ we will denote the limits $f(x)/g(x) \rightarrow 0$, $f(x)/g(x) \rightarrow 1$, and $f(x)/g(x) \rightarrow \infty$ as $x \rightarrow \infty$ by $f \ll g$, $f \sim g$ and $f \gg g$ respectively. We will also assume that there always exists a constant $\alpha > 0$ such that when $\epsilon > 0$ we have
\begin{equation}
\label{eq}
N^{-\alpha-\epsilon} \ll v_2 \ll N^{-\alpha+\epsilon}.
\end{equation}
If $\alpha = 0$ then the mutation rates are too fast to be realistic. To reduce the number of subscripts, we will use $\log x$ for $\log_2 x$. We will use $\rightarrow_d$ to denote convergence in distribution and $\rightarrow_p$ to denote convergence in probability.

One of the two goals of this paper is to find the asymptotic distribution of $\tau(H_1)$ as $N$ approaches infinity. Similar work has been done for the Moran model by Schweinsberg in \cite{Sch} and Durrett, Schmidt and Schweinsberg in \cite{DSS} in which more general results have already been found.  In \cite{Kom}, Komarova makes a connection between the Moran model and the hierarchical one as follows: In the Moran model a mutation may undergo fixation, meaning it spreads throughout the entire population through the birth-death process and all of the cells are the same type. Because the last generation is always removed in the hierarchical model, the only way to get fixation is if the stem cell gets mutated. These are the cases ss and sd. In these cases the mutation will spread throughout the population in $l$ time units. In the Moran model it is also possible that the mutations undergo what is called stochastic tunneling. This is when multiple mutations are acquired before they fixate. This is analogous to daughter cells acquiring two mutations before the stem cell mutates in the hierarchical model. This is the dd case and can also happen in the sd case if the second mutation occurs before the first has time to fixate (which is expected to happen when $\alpha < 1$). The rate at which daughter cells get successful type-1 mutations is given heuristically in \cite{Kom} to be
$$\sum_{i=1}^{l} v_1 2^{i-1}(1-e^{-v_2 (2^{l-i+1}-2)}).$$
One may arrive at this rate by noting that the $i^{th}$ generation has $2^{i-1}$ cells which get type-1 mutations at rate $v_1$. Each of the cells will have $2^{l-i+1}-2$ descendants which live for one time each and get type-2 mutations at rate $v_2$. The distribution of $\tau(H_1)$ will be one part of the main theorem.

Our second goal is to determine which cells obtain the mutations that lead to cancer. The location of the mutations can be essential to the treatment of cancer. As an example, studies of the effects of the drug imatinib on chronic myeloid leukemia have shown that leukemic stem cells will most likely not cause tumors but rather that a tumor is a result of a mutation on one of the daughter cells, see Dingli and Michor \cite{DMi} and Michor \cite{Mic2}. Imatinib treats leukemic daughter cells but not leukemic stem cells. So while using imatinib problems arising from cancer are prevented. However, patients cannot stop treatment because the leukemic stem cells will continue producing new leukemic daughter cells. Therefore, the location of where the mutations occur may play a pivotal role in determining how to treat the cancer. This is the other part of the main theorem in which we determine the limiting distributions of $\sigma(H_1)$ and $\rho(H_1)$.

According to Komarova in \cite{KCh} there are four cases that are particularly interesting from a biological viewpoint.
\begin{enumerate}
\item The null-model. In this model all of the mutation rates are equal. This model is the easiest to work with: $u_1 = u_2 = v_1 = v_2$.
\item Chromosomal instability. In this case the probability of getting a second mutation is greater than that of getting the first mutation, but the rates do not differ between stem cells or daughter cells: $u_1=v_1<u_2=v_2$.
\item Stem cells have a lower mutation rate: $v_1=v_2 > u_1=u_2$.
\item The problem of de-differentiation. In this case the daughter cells have a slower mutation rate. There are two scenarios for this: $v_1=v_2 < u_1=u_2$ or $v_1=u_1=u_2 > v_2$.
\end{enumerate}
We impose the following restriction on the mutation rates: $u_1 \leq u_2$ and $v_1 \leq v_2$. This will cover all of the above scenarios except for the second of the de-differentiation cases.

The following theorem is the goal of this paper. Recall that $\alpha$ is the number from $\eqref{eq}$.
\begin{Theo} \label{mainTheorem}
Recall that all limits are taken as $N$ goes to infinity. Let $X$ be a random variable which has the exponential distribution with parameter 1 and let $Y$ be a random variable which has the Rayleigh distribution so that $P(Y \leq t) = 1 - e^{-t^2 / 2}$ for any $t > 0$.
\begin{enumerate}
\item If $v_1 v_2 \ll 1/N (\log N)^2$ and $v_1 v_2 N \log N \gg u_1$ then $(\alpha \wedge 1) v_1 v_2 N (\log N) \tau(H_1) \rightarrow_d X$. The distribution of $\sigma(H_1)$ converges to the uniform distribution on $((1-\alpha)^+, 1]$ and $\rho(H_1)$ converges in probability to 1.
\item If $1/N (\log N)^2 \ll v_1 v_2 \ll 1/N$ and $\sqrt{v_1 v_2 N} \gg u_1$ then $\sqrt{v_1 v_2 N} \tau(H_1) \rightarrow_d Y$. Both $\sigma(H_1)$ and $\rho(H_1)$ converge in probability to 1.
\item If $v_1 v_2 \gg 1/N$ then $\sqrt{v_1 v_2 N} \tau(H_1) \rightarrow_d Y$. Both $\sigma(H_1)$ and $\rho(H_1)$ converge in probability to 1.
\item If we have the following two conditions:
\begin{itemize}
\item Either $v_1 v_2 \ll 1/N(\log N)^2$ and $u_1 \gg v_1 v_2 N \log N$ or $1/N(\log N)^2 \ll v_1 v_2 \ll 1/N$ and $u_1 \gg \sqrt{v_1 v_2 N}$
\item Both $u_2 \ll 1/\log N$ and $u_2 \ll N v_2$
\end{itemize}
then $u_1 \tau(H_1) \rightarrow_d X$. The probability that the first mutation occurs on the stem cell converges to 1 and $\rho(H_1)$ converges in probability to $\alpha \wedge 1$.
\item If we have the following two conditions:
\begin{itemize}
\item Either $v_1 v_2 \ll 1/N(\log N)^2$ and $u_1 \gg v_1 v_2 N \log N$ or $1/N(\log N)^2 \ll v_1 v_2 \ll 1/N$ and $u_1 \gg \sqrt{v_1 v_2 N}$
\item Either $u_2 \gg 1/\log N$ or $u_2 \gg N v_2$
\end{itemize}
then $u_1 \ll u_2$ implies $u_1 \tau(H_1) \rightarrow_d X$. On the other hand, if $u_1 \sim A u_2$ for some $A \geq 1$ then let $Z$ be an exponentially distributed random variable with parameter $1/A$ which is independent of $X$. Then $u_1 \tau(H_1) \rightarrow_d X+Z$. In either case, the probability that both mutations occur on the stem cell converges to 1.
\end{enumerate}
\end{Theo}

The first three cases in Theorem \ref{mainTheorem} are what happens when the probability that the cancer causing type-1 mutation occurs on a daughter cell converges to 1. Case 4 gives the results for the sd regime and case 5 gives the results for the ss regime.

The third case is a result of fast mutation rates. That is, there will be so many mutations that the probability of two mutations occurring before the model even has time to split once will converge to 1. This reduces to computing the waiting time for the first of $N-1$ Poisson processes to receive two hits.

In both the first and second cases the probability that $\tau(H_1)$ goes to infinity converges to 1. The results of these two cases rely on whether or not $P(\tau(H_1) < \log N)$ converges to 0 or 1. As for the first case, $P(\tau(H_1) < \log N) \rightarrow 0$. The distribution of $\sigma(H_1)$ arises from a balance between the large number of cells in the later generations versus the large number of descendants of cells in the earlier generations.

For the second case $P(\tau(H_1) < \log N) \rightarrow 1$. In this case the mutations occur fast enough that the number of descendants the cells have is not as important. This is why the cancer causing type-1 mutation will occur in the later generations.

In both the first and second cases the second mutation will occur near the top of the crypt. This is because most of the cells are at the top of the crypt. The distribution of $\tau(H_1)$ may be best understood through the following picture:

\setlength{\unitlength}{.3in}
\begin{picture}(5,6.5)(5,5)
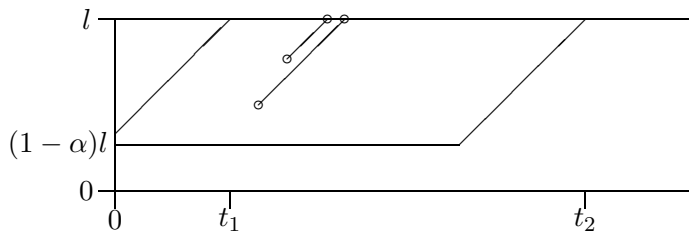
\put(7.5,5.3){\makebox(1,1){\footnotesize An Alternative View of the Model}}
\put (6,6.7){\line(0,1){3.3}}
\put (5.7,7){\line(1,0){10.3}}
\put (5.5,7){\makebox(0,0){0}}
\put (5.5,10){\makebox(0,0){$l$}}
\put (5,7.8){\makebox(0,0){$(1-\alpha)l$}}
\put (5.7,10){\line(1,0){10.3}}
\put (6,6.5){\makebox(0,0){0}}
\put (8,6.5){\makebox(0,0){$t_1$}}
\put (14.2,6.5){\makebox(0,0){$t_2$}}
\put (8.5,8.5){\line(1,1){1.5}}
\put (6,7.8){\line(1,0){6}}
\put (12,7.8){\line(1,1){2.2}}
\put (14.2,6.7){\line(0,1){0.3}}
\put (10,10){\circle{0.15}}
\put (8.5,8.5){\circle{0.15}}
\put (9.7,10){\circle{0.15}}
\put (9,9.3){\circle{0.15}}
\put (9,9.3){\line(1,1){0.7}}
\put (8,6.7){\line(0,1){0.3}}
\put (6,8){\line(1,1){2}}
\end{picture}

The horizontal axis is time and the vertical axis is cell generation. The circles represent cell mutations. The circles within the rectangle represent successful type-1 mutations and the other circles connected to these by a diagonal line which are located at the top of the graph represent their type-2 descendants. The type-2 mutations are at the top of the graph because the later generations are where we expect the type-2 mutations to occur. Likewise, the type-1 mutations are expected to occur in the last $\alpha l$ generations so they lie above the line marked at $(1-\alpha)l$. The infinite rectangle which is bounded between $(1-\alpha)l$ and $l$ vertically and only by 0 on the left horizontally will be dotted within by successful type-1 mutations according to a Poisson process of rate $v_1 v_2 N$. This Poisson process has a uniform rate horizontally because of the time independence of the mutation rates and uniform rate vertically because of the balance between the number of cells in the later generations and the number of descendants of cells in the earlier generations. Notice that in the picture the cancer causing type-1 mutation is not the first successful type-1 mutation.

At times $t_1$ and $t_2$ there are diagonal lines coming out of the top of the graph that enclose a region of the rectangle. To have a type-2 cell by time $t_1$ we must have a successful type-1 mutation in the corresponding triangle. Likewise, to have a type-2 cell by time $t_2$ we must have a successful type-1 mutation in the corresponding quadrilateral. Therefore, the rate at which type-2 mutations occur is converging to the area enclosed in the graph by time $t$ multiplied by the rate of successful type-1 mutations. For the first case we expect the time to get a second mutation to be much larger than $l$ which is represented by $t_2$. Because of this, the waiting time as marked on the graph will go infinitely to the right as $N$ goes to infinity and the quadrilateral will be approximately a rectangle since the missing bottom right corner will have negligible area. This will cause $\tau(H_1)$ to have an exponential distribution. For the second case we expect the time to get a second mutation to be much smaller than $l$ which is represented by $t_1$. In this case the area enclosed in the graph will always be a triangle so that the rate at which we expect to get a type-2 mutation is asymptotic to a function of $t^2$. This results in convergence to the Rayleigh distribution. As $N$ goes to infinity the triangle will be squeezed into the upper left corner.

The convergence of $\sigma(H_1)$ in part 1 is particularly interesting. The location of the cancer causing type-1 mutation is not immediately obvious because the generations with large numbers of cells have fewer descendants on which a type-2 mutations might occur. This result reveals how the high rates of type-1 mutations occurring on later generations balances with the high probability of success of type-1 mutations which occur on earlier generations.

In all but one case $\rho(H_1)$ converges to 1. This happens because of the large number of cells in the later generations. The exception in case 4 is caused by having a low $v_1$ and high $u_1$ and $v_2$. The stem cell will get the first mutation because the daughter cells are slow to acquire type-1 mutations, but the daughter cells acquire type-2 mutations fast enough that they will get a type-2 mutation before all of the daughters inherit the type-1 mutation from the stem cell.

There are many boundary cases and most of them are not included in this paper, where we use the term boundary case to refer to the boundary between two of the conditions. That is, if $v_1 \ll 1/N$ gives one result and $v_1 \gg 1/N$ gives another, we would consider $v_1 \sim A/N$ for some constant $A$ to be a boundary case. If included, the boundary cases would make up the bulk of this paper. One reason for this is that our variables $\{v_1, v_2, u_1, u_2\}$ span a four dimensional space so that the regions will have many boundaries. Moreover, sometimes three regions intersect in the same place. It does not seem that there would be any special difficulties in computing most of these boundary cases and that they could be done with the same methods used in this paper.

The following proposition gives the results for the null-model, including results for the boundary cases.
\begin{Prop} \label{nullProp}
Let $\mu = u_1 = u_2 = v_1 = v_2$. Let $X$ be a random variable which has the exponential distribution with parameter 1. Let $Y$ be a random variable which has the Rayleigh distribution.
\begin{enumerate}
\item If $\mu \ll 1/N\log N$ then $\mu \tau(H_1) \rightarrow_d X$. The probability that the first mutation occurs on the stem cell converges to 1 and $\rho(H_1)$ converges in probability to 1.
\item If $\mu \sim A/N\log N$ then $(1+A) \mu \tau(H_1) \rightarrow_d X$. Let $\xi$ be a Bernoulli random variable such that $P(\xi = 1) = A/(1+A)$ and $P(\xi = 0) = 1/(1+A)$. Let $U$ be a random variable with uniform distribution on $[0,1]$. Then
    $$\sigma(H_1) \rightarrow_d U \xi$$
    and
    $$\rho(H_1) \rightarrow_d \xi + (\alpha \wedge 1)(1-\xi).$$
\item If $1/N\log N \ll \mu \ll 1/\sqrt{N}\log N$ then $(\alpha \wedge 1) \mu^2 N (\log N) \tau(H_1) \rightarrow_d X$. The distribution of $\sigma(H_1)$ converges to a uniform distribution on $((1-\alpha)^+, 1]$ and $\rho(H_1)$ converges in distribution to 1.
\item If $\mu \sim A/\sqrt{N}\log N$ then
    $$\lim P(\tau(H_1)/\log N \leq t) = (1-e^{-A^2 t^2/2})1_{[0,1/2]}(t) + (1-e^{-A^2 t/2 +  A^2/8})1_{(1/2,\infty)}(t).$$
    Let $Z$ be a random variable with density
    $$f(x) = \left(\int_{1-x}^{1/2}A^2e^{-A^2 t^2/2}dt+2e^{-A^2/8}\right)1_{[1/2,1]}(x).$$
    As $N$ goes to infinity $\sigma(H_1)$ converges in distribution to $Z$ and $\rho(H_1)$ converges in probability to 1.
\item If $1/\sqrt{N}\log N \ll \mu \ll 1/\sqrt{N}$ then $\sqrt{N} \mu \tau(H_1) \rightarrow_d Y$. Both $\sigma(H_1)$ and $\rho(H_1)$ converge in distribution to 1.
\item If $\mu \sim A/\sqrt{N}$ then for each fixed time $t > 0$ there exist constants $c$ and $C$ such that $\liminf P(\tau(H_1) \leq t) \geq c > 0$ and $\limsup P(\tau(H_1) \leq t) \leq C < 1$. Both $\sigma(H_1)$ and $\rho(H_1)$ converge in probability to 1.
\item If $1/\sqrt{N} \ll \mu$ then $\sqrt{N} \mu \tau(H_1) \rightarrow_d Y$. Both $\sigma(H_1)$ and $\rho(H_1)$ converge in probability to 1.
\end{enumerate}
\end{Prop}

Parts 1, 3, 5 and 7 of Proposition \ref{nullProp} follow directly from Theorem \ref{mainTheorem}. Parts 2, 4 and 6, the boundary cases, will be done in the last section. In part 2 the cancer causing type-1 mutation may occur on the stem cell or a daughter cell. The event $\xi = 1$ indicates that the cancer causing type-1 mutation occurred on a daughter cell. In part 4 the picture which appears below Theorem \ref{mainTheorem} is especially useful. We create a point process on $[0,\infty) \times [0,1]$ whose points are associated with the mutations. In Lemma \ref{PPLem} we show that the limiting distribution of this point process is a Poisson point process whose intensity is Lebesgue measure on $[0,\infty) \times [1/2,1]$. The main result of part 6 is that when $\mu \sim A/\sqrt{N}$ the mutations will occur in finite time. Because of this, the discreteness of the model cannot be ignored and computing the limit as $N$ goes to infinity becomes difficult. However, this is a degenerate case because the model no longer resembles a colon acquiring mutations.

In the next section we include some known results in probability that will be used throughout the paper. In section 3 we introduce a new model which will be coupled with $H_1$. Theorem \ref{mainTheorem} will be proved with this new model in place of $H_1$ and the coupling will give the results for $H_1$. The fourth section of this paper is devoted to getting results about the dd regime. The fifth section is on results about the sd and ss regimes. In section 6 we determine whether $\tau(H_1)$, $\sigma(H_1)$ and $\rho(H_1)$ will satisfy the results of the  dd, sd or ss regime. The proof of Theorem \ref{mainTheorem} is given at the end of section 6. The last section is a discussion of the boundary cases in the null model and a proof of Proposition \ref{nullProp}.

\section{Preliminaries}
In this section we include some general results about probability which we will make use of in the paper.

\begin{Lemma} \label{forConLem}
If $\{X_n\}_{n=1}^\infty$ is a sequence of nonnegative random variables such that $X_n \rightarrow_d X$ for some finite random variable $X$ and $\{k_n\}_{n=1}^\infty$ is a sequence of positive constants such that $k_n \rightarrow 0$ then $k_n X_n \rightarrow_p 0$.
\end{Lemma}
\begin{proof}
Let $\epsilon > 0$ and $\delta > 0$ be real numbers. Let $M$ be a real number such that the function $F(t) = P(X \leq t)$ is continuous at $M\epsilon$ and $P(X > M\epsilon) < \delta/2.$ Such an $M$ exists because the discontinuities of $F$ are countable. Choose $N_1$ so that if $n \geq N_1$ then $k_n < 1/M$. Choose $N_2$ so that if $n \geq N_2$ then $|P(X_n \leq M\epsilon)-P(X \leq M\epsilon)| < \delta/2.$ Then for $n \geq N_1 \vee N_2$
\nonumber
\begin{align}
P(k_n X_n > \epsilon) & \leq P(X_n/M > \epsilon) \\
& \leq |P(X_n/M > \epsilon)-P(X/M > \epsilon)|+P(X/M > \epsilon) \\
& \leq \delta.
\end{align}
\end{proof}

\begin{Lemma} \label{ConLem}
Let $\{\alpha_n\}_{n=1}^\infty$ and $\{\beta_n\}_{n=1}^\infty$ be sequences of positive numbers which converge to 0. Let $\{X_n\}_{n=1}^\infty$ and $\{Y_n\}_{n=1}^\infty$ be independent sequences of random variables and let $X$ and $Y$ be positive random variables such that $\alpha_n X_n \rightarrow_d X$ and $\beta_n Y_n \rightarrow_d Y$. If $\alpha_n \ll \beta_n$ then $P(X_n \geq Y_n) \rightarrow 1$.
\end{Lemma}
\begin{proof}
First note that $P(X_n \geq Y_n) = P(\alpha_n X_n \geq \alpha_n Y_n)$. Also, $\alpha_n Y_n = (\alpha_n/\beta_n) \beta_n Y_n$ and $\alpha_n / \beta_n \rightarrow 0$ so $\alpha_n Y_n \rightarrow_p 0$ by Lemma \ref{forConLem}.

Let $\delta > 0$ and choose $\epsilon > 0$ such that the function $F(t) = P(X \leq t)$ is continuous at $\epsilon$ and $P(X > \epsilon) > 1-\delta/2$. We can choose $N_1$ such that if $n \geq N_1$ then $P(\alpha_n X_n > \epsilon) > 1-\delta$ by the definition of convergence in distribution. Choose $N_2$ such that if $n \geq N_2$ then $P(\alpha_n Y_n > \epsilon) < \delta$. Then for $n \geq N_1 \vee N_2$
\nonumber
\begin{align}
P(\alpha_n X_n > \alpha_n Y_n) & \geq P(\{\alpha_n X_n > \epsilon\} \cap \{\epsilon > \alpha_n Y_n\}) \\
& = P(\alpha_n X_n > \epsilon) P(\epsilon > \alpha_n Y_n) \\
& > (1-\delta)^2
\end{align}
where $\delta$ can be made arbitrarily small.
\end{proof}

\begin{Lemma} \label{CondLim}
Let $\{A_n\}_{n=0}^\infty$, $\{B_n\}_{n=0}^\infty$ and $\{C_n\}_{n=0}^\infty$ be sequences of events such that $\lim_{n \rightarrow \infty} P(A_n) = a > 0$, $\lim_{n \rightarrow \infty} P(B_n) = 1$ and $\lim_{n \rightarrow \infty} P(C_n) = 0$. Then
$$\lim_{n \rightarrow \infty} P(B_n|A_n) = 1 \mbox{ and } \lim_{n \rightarrow \infty}P(C_n|A_n) = 0.$$
\end{Lemma}
\begin{proof}
First note that
$$\lim_{n \rightarrow \infty} P(A_n \cap C_n) \leq \lim_{n \rightarrow \infty} P(C_n) = 0.$$
For $n$ large enough $P(A_n)$ is never 0, so $\lim_{n \rightarrow \infty} P(C_n|A_n) = \lim_{n \rightarrow \infty} P(C_n \cap A_n)/P(A_n) = 0.$

Likewise, $\lim_{n \rightarrow \infty} P(B_n^C) = 0$ so $\lim_{n \rightarrow \infty} P(B_n^C|A_n) = 0$. Therefore, the same reasoning yields $\lim_{n \rightarrow \infty} P(B_n|A_n) = 1$.
\end{proof}

\section{A Useful Model}
There is a similar model $H_2$ which will be coupled with model $H_1$. This model is the same as $H_1$ except in the way the daughter cells acquire type-2 mutations. Label the daughter cells $D_1, D_2, \dots D_{N-1}$. In model $H_2$ each daughter cell $D_i$ has a counter $C_i$ starting at 0 and is acted on by a sequence of Poisson processes $\{P_n^i\}_{n=1}^{\infty}$ which determine the type-2 mutations. All of the Poisson processes are independent. In this model, when a type-1 mutation occurs on a daughter cell $D_i$ it increases the counter $C_i$ by 1. This is considered as a type-1 mutation. If a type-1 mutation increases the counter to $n$, it is the $n^{th}$ type-1 mutation on the cell. When the counter $C_i$ has reached $n$, any type-2 mutations that would occur according to the Poisson processes $P_1^i, P_2^i, \dots P_n^i$ are accepted as type-2 mutations on cell $D_i$. Any type-2 mutations that would occur according the the Poisson processes $P_{n+1}^i, P_{n+2}^i, \dots$ are rejected. If a type-2 mutation occurs on cell $D_i$ as a result of the Poisson process $P_n^i$, then the $n^{th}$ type-1 mutation according to $C_i$ is considered to be successful. If the first type-2 mutation on a cell is a result of the Poisson process $P_n^i$, then the $n^{th}$ type-1 mutation according to $C_i$ is the cancer causing type-1 mutation. However, a type-1 mutation on the stem cell does not have a counter. Once a type-1 mutation has spread from the stem cell to a daughter cell the daughter cell can no longer accumulate type-1 mutations and the model is the same as model $H_1$.

There is an extra convenience embedded in the model $H_2$. We can consider the $N-1$ Poisson processes that mark the type-1 mutations on the individual daughter cells as one Poisson process which marks the mutations on the population of daughter cells whose measure is time independent. The rate of the Poisson process is $v_1 (N-1)$ and when a type-1 mutation occurs it occurs on any particular cell with probability $1/(N-1)$. In the Hierarchical model, $H_1$, the mutations are suppressed on type-1 cells so that the rate at which the population of daughter cells is acquiring type-1 mutations depends on how many type-1 cells there are at the time.

We couple $H_1$ and $H_2$ by allowing the same Poisson processes to mark the mutations on the cells within each model. The Poisson processes that mark the stem cells are the same. If a daughter cell has inherited a type-1 mutation from the stem cell then the Poisson processes marking type-2 mutations on the cell are the same in each model. The Poisson processes marking type-1 mutations on daughter cells are the same. The Poisson processes marking type-2 mutations on daughter cells in model $H_1$ are the same as the Poisson processes $P_1^i$ in model $H_2$ so long as the daughter cells did not inherit their type-1 mutations from the stem cell. There are no analogous Poisson processes in model $H_1$ for the $N-1$ sequences of Poisson processes $P_2^i, P_3^i, \dots$ in model $H_2$.

\begin{Lemma} \label{coupleLem}
Let the Poisson processes in models $H_1$ and $H_2$ be coupled as described above. Then $P(\tau(H_1) = \tau(H_2))$, $P(\rho(H_1) = \rho(H_2))$ and $P(\sigma(H_1) = \sigma(H_2))$ all converge to 1.
\end{Lemma}
\begin{proof}
A type-2 mutation which occurs in model $H_2$ but not in $H_1$ is a result of the rejection of the type-1 mutation in model $H_1$ that has led to the type-2 mutation in $H_2$. This mutation is rejected because the cell on which the type-1 mutation was supposed to occur already was a type-1 cell. Any cell has at most $\log N$ ancestors, so the probability that a type-1 mutation is rejected is less than $e^{-v_1 \log N}$. Therefore, the probability of rejecting the type-1 mutation that causes the first type-2 mutation is converging to 0 as long as $v_1 \ll 1/\log N$. That is, if we number the cells $1,2,\dots,N$ and let $A_i$ be the event that the cancer causing mutation happens on cell $i$,
$$P(\tau(H_1) \neq \tau(H_2)) = \sum_{i=1}^N P(\tau(H_1) \neq \tau(H_2)|A_i)P(A_i) \leq \sum_{i=1}^N e^{-v_1 \log N} P(A_i) = e^{-v_1 \log N} \rightarrow 0.$$

If we do not have $v_1 \ll 1/\log N$ then because $v_2 \geq v_1$ we do not have $v_2 \ll 1/\log N$ either, with contradicts equation (\ref{eq}). Hence we only need to consider the case $v_1 \ll 1/\log N$.
\end{proof}

The rest of the work in proving Theorem \ref{mainTheorem} is in proving Theorem \ref{mainTheorem} with $H_2$ in place of $H_1$. Once this is done Theorem \ref{mainTheorem} follows from Lemma \ref{coupleLem}.

\section{The dd regime}
To understand the behavior in the dd regime, we consider a new model which is the same as $H_2$ except that mutations only occur on daughter cells. That is, there are no Poisson processes that cause mutations on the stem cells. This new model will be called model $M_1$. The purpose of this section is to prove Proposition \ref{ddRegime}.
\begin{Prop} \label{ddRegime}
Let $X$ be a random variable which has the exponential distribution with parameter 1. Let $Y$ be a random variable which has the Rayleigh distribution.
\begin{enumerate}
\item If $v_1 v_2 \ll 1/N (\log N)^2$ then $(\alpha \wedge 1) v_1 v_2 N (\log N) \tau(M_1) \rightarrow_d X$. The distribution of $\sigma(M_1)$ converges to a uniform distribution on $((1-\alpha)^+, 1]$ and $\rho(M_1)$ converges in probability to 1.
\item If $1/N (\log N)^2 \ll v_1 v_2 \ll 1/N$ then $\sqrt{N v_1 v_2} \tau(M_1) \rightarrow_d Y$. Both $\sigma(M_1)$ and $\rho(M_1)$ converge in probability to 1.
\item If $v_1 v_2 \gg 1/N$ then $\sqrt{N v_1 v_2} \tau(M_1) \rightarrow_d Y$. Both $\sigma(M_1)$ and $\rho(M_1)$ converge in probability to 1.
\end{enumerate}
\end{Prop}

\begin{Lemma} \label{Kgen}
For any positive integer $k < l$ we have $P(\rho(M_1) \geq (l-k)/l) > 1-1/2^k$.
\end{Lemma}
\begin{proof}
Let $Y$ be the number of generations between the cancer causing type-1 mutation and the first type-2 mutation. Then $Y \in \{1,2, \dots, l\}$. Because there are only $l$ generations, if the second mutation occurs $l-k$ generations or more after the first then it must be in the last $k$ generations. So $P(\rho(M_1) \geq (l-k)/l|Y \in \{l-k, l-k+1, \dots, l\}) = 1$. If we condition on the event that $Y=k$, then the probability that the cancer causing type-1 mutation occurs on any cell in generations $1, 2, \dots, l-Y$ is equally likely. This is because the descendants of the cells are independent and identically distributed. The last $k$ of the $l-Y$ generations always make up at least a fraction of $1-1/2^k$ cells, so we have $P(\rho(M_1) \geq (l-k)/l|Y \in \{1, 2, \dots, l-k-1\}) > 1-1/2^k$ where we get a strict inequality because we do not count the stem cell. The result follows.
\end{proof}

It is important to notice in the above lemma that we do not need $N \rightarrow \infty$. We can see from the above lemma that $P(\rho(M_1) \geq (l-k)/l) > 1-1/2^k$ holds for any $N$ so it remains valid as $N \rightarrow \infty$.

\begin{Cor} \label{ddsigma2}
As $N$ goes to infinity, $\rho(M_1)$ will converge to 1 in probability.
\end{Cor}

\begin{Lemma} \label{longlem}
Let $(\beta_1, \beta_2] \subset (0,1]$ and let $C$ and $C'$ be a positive constants. Then
$$\sum_{i \in \N \cap (l\beta_1, l\beta_2]} v_1 2^{i-1}(1-e^{-C v_2(2^{l-i+1}-C')}) \sim C (\beta_2-\beta_1\vee(1-\alpha))^+ v_1 v_2 N \log N.$$
\end{Lemma}
\begin{proof}
We will first define some notation for this proof for the sake of readability. Let $I \subset \R$. We define
$$I^* := I \cap (l\beta_1, l\beta_2] \cap \N.$$

First we can do the case when $\alpha \geq 1$. Let $0 < \epsilon < 1$ and break the sum into two parts,
$$\frac{\sum_{i \in (l\beta_1,l\beta_2]^*} v_1 2^{i-1}(1-e^{-C v_2(2^{l-i+1}-C')})}{v_1 v_2 2^l l} = \frac{(\sum_{i\in[1,l\epsilon]^*} + \sum_{i\in(l\epsilon,l]^*}) 2^{i-1}(1-e^{-C v_2(2^{l-i+1}-C')})}{v_2 2^l l}.$$
If we use the upper bound $1-e^{-C v_2(2^{l-i+1}-C')} \leq C v_2(2^{l-i+1}-C') \leq C v_2 2^{l-i+1}$ then
$$0 \leq \frac{\sum_{i\in[1,l\epsilon]^*} 2^{i-1}(1-e^{-C v_2(2^{l-i+1}-C')})}{v_2 2^l l} \leq C (\beta_2 \wedge \epsilon - \beta_1)^+.$$

As for the second sum, the same upper bound yields
$$\frac{\sum_{i \in (l\epsilon,l]^*} 2^{i-1}(1-e^{-C v_2 (2^{l-i+1}-C')})}{v_2 2^l l} \leq C(\beta_2-\beta_1 \vee \epsilon)^+.$$
From the second order Taylor expansion we get a lower bound of
$$1-e^{-C(2^{l-i+1}-C')} \geq C v_2 (2^{l-i+1}-C') - \frac{1}{2}C^2 v_2^2 (2^{l-i+1}-C')^2.$$
We can show that this sum will go up to $1-\epsilon$ by breaking the sum into 5 parts,
$$2^{i-1}(1-e^{-C v_2 (2^{l-i+1}-C')}) \geq 2^l C v_2 - 2^i C C' v_2 - 2^{2l-i}C^2 v_2^2+2^{l}C^2 C' v_2^2-2^{i-2} C^2 (C')^2 v_2^2.$$

We get the following computations for each of the five individual sums:
\nonumber
\begin{align}
\sum_{i \in (l\epsilon,l]^*} & 2^l C v_2/v_2 2^l l \rightarrow C (\beta_2 - \beta_1 \vee \epsilon)^+. \\
\sum_{i \in (l\epsilon,l]^*} & 2^i C C' v_2/v_2 2^l l \leq C C' 2^{l+1}/(2^l l) \rightarrow 0. \\
\sum_{i \in (l\epsilon,l]^*} & 2^{i-2} C^2 (C')^2 v_2^2/v_2 2^l l \leq 2 C^2 C'^2 v_2/l \rightarrow 0. \\
\sum_{i \in (l\epsilon,l]^*} & 2^l C^2 C' v_2^2/v_2 2^l l \leq C^2 C' v_2 \rightarrow 0. \\
\sum_{i \in (l\epsilon,l]^*} & 2^{2l-i} C^2 v_2^2/v_2 2^l l = C^2 v_2 2^l (\sum_{i=\lceil l\epsilon \rceil}^{l} 2^{-i})/l \leq C^2 v_2 2^{l(1-\epsilon) }\rightarrow 0
\end{align}
so long as $v_2 \ll 1/2^{l(1-\epsilon)} = N^{-1+\epsilon}$ which will hold since this is the case $\alpha \geq 1$.

So we have
$$C (\beta_2-\beta_1 \vee \epsilon)^+ \leq \liminf \frac{\sum_{i \in (l\beta_1,l\beta_2]^*} v_1 2^{i-1}(1-e^{-C v_2(2^{l-i+1}-C')})}{v_1 v_2 2^l l}$$
and because $(\beta_2-\beta_1 \vee \epsilon)^+ + (\beta_2 \wedge \epsilon - \beta_1)^+ = \beta_2 - \beta_1$ we also have
$$\limsup \frac{\sum_{i \in (l\beta_1,l\beta_2]^*} v_1 2^{i-1}(1-e^{-C v_2(2^{l-i+1}-C')})}{v_1 v_2 2^l l} \leq C (\beta_2 - \beta_1).$$
Since $\epsilon$ may be made arbitrarily small we have finished the case for $\alpha \geq 1$.

Now let $0 < \alpha < 1$ and let $\epsilon > 0$ be small enough so that $0 < 1-\alpha-\epsilon < 1-\alpha + \epsilon < 1$. We now break the sum into three pieces,
$$\frac{(\sum_{i \in [1,l(1-\alpha-\epsilon))^*} + \sum_{i \in [l(1-\alpha-\epsilon),l(1-\alpha+\epsilon)]^*} + \sum_{i \in (l(1-\alpha+\epsilon),l]^*}) 2^{i-1}(1-e^{-C v_2(2^{l-i+1}-C')})}{v_2 2^l l}.$$
We can consider each of these three sums individually.

As for the middle sum, we only need the bound
$$0 \leq \frac{\sum_{i \in [l(1-\alpha-\epsilon),l(1-\alpha+\epsilon)]^*} 2^{i-1}(1-e^{-C v_2(2^{l-i+1}-C')})}{v_2 2^l l} \leq 2 C \epsilon$$
which follows by the upper bound $1-e^{-C v_2 (2^{l-i+1}-C')} \leq C v_2 2^{l-i+1}$.

One can apply similar computations as in the case when $\alpha = 1$ to obtain the following:
$$\frac{\sum_{i \in (l(1-\alpha+\epsilon),l]^*} 2^{i-1}(1-e^{-C v_2(2^{l-i+1}-C')})}{v_2 2^l l} \rightarrow C (\beta_2-\beta_1 \vee (1-\alpha+\epsilon))^+.$$

For the first sum, note that $1-e^{-C v_2(2^{l-i+1}-C')} \leq 1$. This gives the bound
\nonumber
\begin{align}
0 & \leq \sum_{i \in [1,l(1-\alpha-\epsilon))^*} \frac{2^{i-1}(1-e^{-C v_2 (2^{l-i+1}-C')})}{v_2 2^l l} \\
 & \leq \sum_{i \in [1,l(1-\alpha-\epsilon))^*} \frac{2^{i-1}}{v_2 2^l l} \\
 & \leq \frac{2^{l(1-\alpha-\epsilon)}}{v_2 2^l l} \rightarrow 0.
\end{align}
The convergence is a result of the definition of $\alpha$, namely that $v_2 \gg N^{-\alpha-\epsilon}\log^{-1}N$.

Combining the three sums yields
$$C (\beta_2-\beta_1 \vee (1-\alpha+\epsilon))^+ \leq \liminf \frac{\sum_{i \in (l\beta_1,l\beta_2]^*} v_1 2^{i-1}(1-e^{-C v_2(2^{l-i+1}-C')})}{l v_1 v_2 2^l}$$
and
$$\limsup \frac{\sum_{i \in (l\beta_1,l\beta_2]^*} v_1 2^{i-1}(1-e^{-C v_2(2^{l-i+1}-C')})}{l v_1 v_2 2^l} \leq C (\beta_2-\beta_1 \vee (1-\alpha+\epsilon))^+ + 2 C \epsilon.$$
Again, $\epsilon$ may be arbitrarily small which gives the result.
\end{proof}

\begin{Cor} \label{longCor}
For any time $t$, the rate at which successful type-1 mutations occur is asymptotic to $(\alpha \wedge 1)v_1 v_2 N \log N$.
\end{Cor}
\begin{proof}
For $1 \leq i \leq l$ there are $2^{i-1}$ cells in generation $i$. Each of these cells is getting type-1 mutations at rate $v_1$. The cells in generation $i$ have $2^{l-i+1}-2$ descendants. If the cell splits as soon as it becomes a type-1, the probability that none of its descendants get a type-2 mutation is $e^{-v_2(2^{l-i+1}-2)}$. On the other hand, after a cell gets a type-1 mutation it could live for at most 1 time unit until it splits. If this is the case, then the probability that neither the cell that receives the type-1 mutation nor any of its descendants get a type-2 mutation is $e^{-v_2(2^{l-i+1}-1)}$. If we let $R(t)$ be the rate at which the successful type-1 mutations occur at time $t$, then for any time $t$ we have
\begin{align*}
1 & = \lim \frac{\sum_{i=1}^{l} v_1 2^{i-1}(1-e^{-v_2(2^{l-i+1}-2)})}{(\alpha \wedge 1) v_1 v_2 N \log N} \leq \liminf \frac{R(t)}{(\alpha \wedge 1) v_1 v_2 N \log N} \\
& \leq \limsup \frac{R(t)}{(\alpha \wedge 1) v_1 v_2 N \log N} \leq \lim \frac{\sum_{i=1}^{l} v_1 2^{i-1}(1-e^{-v_2(2^{l-i+1}-1)})}{(\alpha \wedge 1) v_1 v_2 N \log N} = 1,
\end{align*}
where the limits are results of Lemma \ref{longlem}.
\end{proof}

\begin{Lemma} \label{bigProp}
If $v_1 v_2 \ll 1/N (\log N)^2$ then the distribution of $\sigma(M_1)$ converges to the uniform distribution on $((1-\alpha)^+,1]$.
\end{Lemma}
\begin{proof}
Let $X_1$ be the time at which the cancer causing mutation occurs and let $Y_1$ be the time at which the first successful type-1 mutation occurs. By Corollary \ref{longCor} we have that the random variable $(\alpha \wedge 1)v_1v_2N(\log N)Y_1$ is converging in distribution to an exponentially distributed random variable with parameter 1. Let $Y_2$ be the time it takes to get the second successful type-1 mutation after the first and let $X_2 = \tau(M_2) - Y_1$. As a result of Corollary \ref{longCor} again, $(\alpha \wedge 1)v_1v_2N(\log N)Y_2$ converges in distribution to an exponentially distributed random variable with parameter 1. Then because a type-2 mutation must occur within $\log N$ time after a successful type-1 mutation on a daughter cell we have
$$P(Y_2 < X_2) \leq P(Y_2 < \log N) = P((\alpha \wedge 1)v_1v_2N(\log N)Y_2 < (\alpha \wedge 1)v_1v_2N(\log N)^2) \rightarrow 0.$$
Moreover, $P(Y_2 \geq X_2) \leq P(Y_1 = X_1)$ so $P(Y_1 = X_1) \rightarrow 1$. Therefore, it is enough to find the distribution of the first successful type-1 mutation.

Each generation $i$ with $1 \leq i \leq l$ is getting successful type-1 mutations independently at a rate bounded between $v_1 2^{i-1}(1-e^{-v_2(2^{l-i+1}-2)})$ and $v_1 2^{i-1}(1-e^{-v_2(2^{l-i+1}-1)})$ for any time $t$. Therefore, for a fixed $N$ and $i$, the probability that the first successful type-1 mutation occurs on generation $i$ is between
$$\frac{v_12^{i-1}(1-e^{-v_2(2^{l-i+1}-2)})}{\sum_{j=1}^l v_1 2^{j-1}(1-e^{-v_2(2^{l-j+1}-1)})}$$
and
$$\frac{v_1 2^{i-1}(1-e^{-v_2(2^{l-i+1}-1)})}{\sum_{j=1}^l v_1 2^{j-1}(1-e^{-v_2(2^{l-j+1}-2)})}.$$
Let $\beta \in [0,1]$. Using the notation and result from Lemma \ref{longlem},
$$\limsup P(\sigma(M_1) \leq \beta) \leq \limsup \frac{\sum_{i \in (0,l\beta]^*} v_12^{i-1}(1-e^{-v_2(2^{l-i+1}-1)})}{\sum_{i \in (0,l]^*} v_1 2^{j-1}(1-e^{-v_2(2^{l-j+1}-2)})} = \frac{(\beta-(1-\alpha)^+)^+}{\alpha \wedge 1}$$
and
$$\liminf P(\sigma(M_1) \leq \beta) \geq \liminf \frac{\sum_{i \in (0,l\beta]^*} v_12^{i-1}(1-e^{-v_2(2^{l-i+1}-2)})}{\sum_{i \in (0,l]^*} v_1 2^{j-1}(1-e^{-v_2(2^{l-j+1}-1)})} = \frac{(\beta-(1-\alpha)^+)^+}{\alpha \wedge 1}.$$
\end{proof}

\begin{Lemma} \label{ddExp}
If $v_1 v_2 \ll 1/N (\log N)^2$ then $(\alpha \wedge 1) v_1 v_2 N (\log N) \tau(M_1) \rightarrow_d X$ where $X$ is an exponential random variable with parameter $1$.
\end{Lemma}
\begin{proof}
Let $X_1$ be the time at which the cancer causing type-1 mutation occurs and let $X_2 = \tau(M_1) - X_1$. From the proof of Lemma \ref{bigProp} we know that the probability that the first successful type-1 mutation is the cancer causing mutation is converging to 1. By Corollary \ref{longCor} we know that the rate of successful type-1 mutations is approaching $(\alpha \wedge 1) v_1 v_2 N \log N$. This gives us that $(\alpha \wedge 1) v_1 v_2 N (\log N) X_1$ is converging in distribution to an exponentially distributed random variable with parameter 1.

Let $\epsilon > 0$. Due to apoptosis $X_2$ is bounded above by $\log N$ so
$$P((\alpha \wedge 1) v_1 v_2 N (\log N) X_2 > \epsilon) \rightarrow 0.$$
In other words, $(\alpha \wedge 1) v_1 v_2 N (\log N) X_2 \rightarrow_p 0$. Then
$$(\alpha \wedge 1) v_1 v_2 N (\log N) \tau(M_1) = (\alpha \wedge 1) v_1 v_2 N (\log N) (X_1 + X_2) \rightarrow_d X.$$
\end{proof}

Combining the results of Corollary \ref{ddsigma2} and Propositions \ref{bigProp} and \ref{ddExp} we have part 1 of Proposition \ref{ddRegime}. For the next two proofs we note that Corollary \ref{ddsigma2} already gives us that $\rho(M_1)$ converges to 1 in probability.

\begin{proof}[Proof of part 2 of Proposition \ref{ddRegime}.]
Consider generation $i$ for some $i \in \{1, \dots, l\}$ at time 0. The total number of descendants of the cells in generation $i$ is $2^{l-i+1}-2$. However, if $t < l-i$ then the total number of descendants by time $t$ is between $2^{t-1}$ and $2^{t+1}$. At each integral time unit there is a new collection of cells in generation $i$. We can consider a sequence of collections of cells where the first element in the sequence is the collection of cells in generation $i$ during time $[0,1)$, the second element is the collection of cells in generation $i$ during time [1,2), and so on. Because the Poisson processes marking the type-1 mutations in model $H_2$ are independent of the type-1 mutations that have already occurred we can consider the sequence of cells in these generations and their descendants that occur over time to be independent. Also, the random variables denoting the times at which type-2 mutations occur as a result of type-1 mutations on the cells in generation $i$ would be identically distributed if we were to start each new collection of cells in generation $i$ at time 0. If $t < l-i$ then by time $t$ the number of cells which will have descended from the $j^{th}$ element in the sequence will be between $2^{t-1-j}$ and $2^{t+1-j}$ for $j \leq \lfloor t \rfloor$. If we sum over all of the terms in the sequence which have appeared by time $t$, the total number of cells which have descended from a cell in generation $i$ (including those which have already undergone apoptosis) will be between
$$\sum_{j=0}^{\lfloor t \rfloor} 2^{t-1-j} \geq 2^t-1$$
and
$$\sum_{j=0}^{\lfloor t \rfloor} 2^{t+1-j} \leq 2^{t+2}-1.$$
If $t \geq l-i$ then by time $t$ the total number of cells which will have descended from a cell in generation $i$ will be between
$$\sum_{j=0}^{l-i}2^{l-i-j-1}+(t-l+i)(2^{l-i+1}-2) = 2^{l-i}-1+(t-l+i)(2^{l-i+1}-2)$$
and
$$\sum_{j=0}^{l-i}2^{l-i-j+1}+(t-l+i)(2^{l-i+1}-2) = 2^{l-i+2}-1+(t-l+i)(2^{l-i+1}-2).$$
Recall that there are always $2^{i-1}$ cells in generation $i$ which are acquiring type-1 mutations at rate $v_1$. If we multiply the rate of type-1 mutations on generation $i$ by the probability that such a mutation is successful, we find that the type-2 mutations that occur as a result of successful type-1 mutations that occur on generation $i$ occur according to a Poisson process that has intensity measure between
$$2^{i-1}v_1(1-e^{-v_2(2^t-1)})\mbox{ and }2^{i-1} v_1 (1-e^{-v_2 (2^{t+2}-1)})$$
if $t < l-i$ and
$$2^{i-1}v_1 (1-e^{-v_2(2^{l-i}-1+(t-l+i)(2^{l-i+1}-2))}) \mbox{ and }2^{i-1} v_1 (1-e^{-v_2 (2^{l-i+2}-1+(t-l+i)(2^{l-i+1}-2))})$$
if $t \geq l-i$.

First we concentrate on the upper bound. For $N$ large enough we will have $t < \sqrt{v_1 v_2 N} \log N$ for any real number $t$ by the hypothesis $1/N (\log N)^2 \ll v_1 v_2$. Let $t/\sqrt{v_1 v_2 N} < l$. Then
$$P(\tau(M_2) \leq \frac{t}{\sqrt{v_1 v_2 N}}) = 1-e^{-f(N,t)}$$
where by summing over the generations and using the fact that $1-e^{-x} \leq x$ we obtain
\begin{align*}
f(N,t) & \leq \sum_{0 \leq i < l-\frac{t}{\sqrt{v_1 v_2 N}}}2^{i-1} v_1 (1-e^{-v_2 (2^{t/\sqrt{v_1 v_2 N}+2}-1)}) \\
& \hspace{20pt} + \sum_{l-\frac{t}{v_1 v_2 N} \leq i \leq l}2^{i-1} v_1 (1-e^{-v_2 (2^{l-i+2}-1+(t/\sqrt{v_1 v_2 N}-l+i)(2^{l-i+1}-2))}) \\
& \leq \sum_{0 \leq i < l-\frac{t}{\sqrt{v_1 v_2 N}}}2^{i-1} (2^{t/\sqrt{v_1 v_2 N}+2}-1)v_1 v_2 \\
& \hspace{20pt} + \sum_{l-\frac{t}{v_1 v_2 N} \leq i \leq l}2^{i-1}\left(2^{l-i+2}-1+\left(\frac{t}{\sqrt{v_1 v_2 N}}-l+i\right)(2^{l-i+1}-2)\right)v_1 v_2.
\end{align*}
As for the first sum,
\begin{align*}
\sum_{0 \leq i < l-\frac{t}{\sqrt{v_1 v_2 N}}}2^{i-1}(2^{t/\sqrt{v_1 v_2 N}+2}-1)v_1 v_2 & \leq \frac{1}{2}(2^{t/\sqrt{v_1 v_2 N}+2}-1) (2^{l-t/\sqrt{v_1 v_2 N}+1}-1)v_1 v_2 \\
& \leq 2^{l+2} v_1 v_2 \rightarrow 0.
\end{align*}
As for the second sum, we first compute
$$\sum_{l-\frac{t}{\sqrt{v_1 v_2 N}} \leq i \leq l}2^{i-1}(2^{l-i+2}-1)v_1 v_2 \leq 2^{l+2} v_1 v_2 \frac{t}{\sqrt{v_1 v_2 N}} \rightarrow 0.$$
Lastly,
\begin{align*}
\sum_{l-\frac{t}{\sqrt{v_1 v_2 N}} \leq i \leq l}2^{i-1}\left(\frac{t}{\sqrt{v_1 v_2 N}}-l+i\right)(2^{l-i+1}-2)v_1 v_2 & \leq 2^l v_1 v_2 \sum_{l-\frac{t}{\sqrt{v_1 v_2 N}} \leq i \leq l} \left(\frac{t}{\sqrt{v_1 v_2 N}}-l+i\right) \\
& \leq \frac{2^l v_1 v_2}{2} \left(\frac{t}{\sqrt{v_1 v_2 N}}+1\right)^2  \\
& \rightarrow \frac{t^2}{2}.
\end{align*}
Therefore, $\limsup P(\sqrt{v_1 v_2 N}\tau(M_2) \leq t) \leq 1-e^{-t^2/2}$.

As for the lower bound, we have
\begin{align*}
f(N,t) & \geq \sum_{0 \leq i < l-\frac{t}{\sqrt{v_1 v_2 N}}} 2^{i-1} v_1 (1-e^{-v_2(2^{t/\sqrt{v_1 v_2 N}}-1)}) \\
& \hspace{20pt} +\sum_{l-\frac{t}{\sqrt{v_1 v_2 N}} \leq i \leq l}2^{i-1} v_1 (1-e^{-v_2 (2^{l-i} - 1 + (t/\sqrt{v_1 v_2 N}-l+i)(2^{l-i+1}-2))}) \\
& \geq \sum_{l-\frac{t}{\sqrt{v_1 v_2 N}} \leq i \leq l}2^{i-1} v_1 (1-e^{-v_2 (t/\sqrt{v_1 v_2 N}-l+i)(2^{l-i+1}-2)}).
\end{align*}
Using the bound $1-e^{-x} \geq x-x^2/2$ we have
$$\sum_{l-t/\sqrt{v_1 v_2 N} \leq i \leq l} 2^{i-1} v_1 (1-e^{-v_2 (t/\sqrt{v_1 v_2 N}-l+i)(2^{l-i+1}-2)})$$
will be greater than or equal to the sum over $i \in [l-t/\sqrt{v_1 v_2 N}, l]$ of
$$2^{i-1} v_1 \left(v_2 \left(\frac{t}{\sqrt{v_1 v_2 N}}-l+i\right)(2^{l-i+1}-2)-v_2^2 \left(\frac{t}{\sqrt{v_1 v_2 N}}-l+i\right)^2(2^{l-i+1}-2)^2/2\right).$$
First consider
$$\sum_{l-t/\sqrt{v_1 v_2 N} \leq i \leq l}2^{i-1} v_1 v_2^2 \left(\frac{t}{\sqrt{v_1 v_2 N}}-l+i\right)^2 \frac{(2^{l-i+1}-2)^2}{2}.$$
This sum is bounded between 0 and $\sum_{l-t/\sqrt{v_1 v_2 N} \leq i \leq l} v_2 t^2 2^{l-i}.$ Let $0 < \epsilon < \alpha$. For $N$ large enough we have $t < \sqrt{v_1 v_2 N} l (\alpha-\epsilon)$ which is equivalent to $l(1-\alpha-\epsilon) < l-t/\sqrt{v_1 v_2 N}.$ So for $N$ large enough we have
$$\sum_{l-t/\sqrt{v_1 v_2 N} \leq i \leq l}v_2 t^2 2^{l-i} \leq \sum_{l-l(1-\alpha+\epsilon) \leq i \leq l}v_2 t^2 2^{l-i} \leq l v_2 N^{\alpha-\epsilon} \rightarrow 0.$$

This leaves us to show
$$\liminf \sum_{l-t/\sqrt{v_1 v_2 N} \leq i \leq l}2^{i-1} v_1 v_2 \left(\frac{t}{\sqrt{v_1 v_2 N}}-l+i\right)(2^{l-i+1}-2) \geq \frac{t^2}{2}.$$
Let $j \in \N$ and $t > 0$. By our assumptions, for large enough values of $N$ we will have $j < t/\sqrt{v_1 v_2 N} < \log N$. Notice that if $i \leq l-j$ then $2^{l-i+1}-2 \geq (1-2^{-j})2^{l-i+1}$, so
\begin{align*}
\sum_{l-t/\sqrt{v_1 v_2 N} \leq i \leq l} & 2^{i-1} v_1 v_2 \left(\frac{t}{\sqrt{v_1 v_2 N}}-l+i\right)(2^{l-i+1}-2) \\
& \geq \sum_{l-t/\sqrt{v_1 v_2 N} \leq i \leq l-j}2^{i-1} v_1 v_2 \left(\frac{t}{\sqrt{v_1 v_2 N}}-l+i\right)(1-2^{-j})2^{l-i+1}.
\end{align*}
Because $j$ is fixed we have
$$\sum_{l-j \leq i \leq l}2^{i-1} v_1 v_2 \left(\frac{t}{\sqrt{v_1 v_2 N}}-l+i\right)(1-2^{-j})2^{l-i+1} \rightarrow 0$$
since each of the summands converges to 0. Therefore, we can add this sum without changing the limit. This gets us a lower bound of
\begin{align*}
\liminf \sum_{l-t/\sqrt{v_1 v_2 N} \leq i \leq l}2^l v_1 v_2 \left(\frac{t}{\sqrt{v_1 v_2 N}}-l+i\right)(1-2^{-j}) \geq \frac{t^2}{2}(1-2^{-j}).
\end{align*}
We chose $j$ to be any natural number, so $\liminf P(\sqrt{v_1 v_2 N} \tau(M_2) \leq t) \geq 1-e^{-t^2/2}$.

The above two bounds establish that $P(\sqrt{v_1 v_2 N} \tau(M_2) \leq t) \rightarrow 1-e^{-t^2/2}$ for any $t \geq 0$. This leaves us to show that $\sigma(M_1)$ converges in probability to 1. First note that for any $\epsilon > 0$ we have
$$P(\tau(M_1) \leq \epsilon \log N) = P(\sqrt{N v_1 v_2} \tau(M_1) \leq \sqrt{N v_1 v_2} \epsilon \log N) \rightarrow 1$$
which follows because the distribution of $\sqrt{N v_1 v_2} \tau(M_1)$ is converging to the Rayleigh distribution and $\sqrt{N v_1 v_2} \epsilon \log N$ is converging to 0. Let $\delta > 0$. By Corollary \ref{ddsigma2} we know that $\rho(M_1)$ converges in probability to 1 so that as $N$ goes to infinity, $P(\rho(M_1) > 1-\delta) \rightarrow 1$. If $\sigma(M_1) < 1-2\delta$ and $\rho(M_1) > 1-\delta$ then $\tau(M_1) > \delta \log N$. Because $P(\tau(M_1) > \delta \log N) \rightarrow 0$ we must also have $P(\sigma(M_1) < 1-2\delta) \rightarrow 0$ where $\delta > 0$ was arbitrary. Then $P(1-\sigma(M_1) > 2\delta) \rightarrow 0$ for any $\delta > 0$ so $\sigma(M_1) \rightarrow_p 1$.
\end{proof}

\begin{proof}[Proof of part 3 of Proposition \ref{ddRegime}.]
We shall make use of the following well known fact: If $\{a_n\}_{n=1}^{\infty}$ is a sequence of real numbers such that $a_n \rightarrow a$, then
$$\lim_{n \rightarrow \infty}(1-\frac{a_n}{n})^{n-1} = e^a.$$

Before time 1 the cells never split and there is no apoptosis. If we ignore the splitting and apoptosis and consider how long it takes for a cell to acquire two mutations under the mutation mechanism alone then we have $N-1$ cells acquiring mutations independently. For any individual cell, the time it takes to acquire two mutations will have the same distribution as the sum of two independent exponentially distributed random variables with parameters $v_1$ and $v_2$. If we denote the time until cell $i$ has a type-2 mutation by $T_i$ and assume $v_1 \neq v_2$ then
$$P(T_i \leq t) = 1-\frac{v_2e^{-v_1 t}-v_1e^{-v_2 t}}{v_2-v_1}.$$
There are $N-1$ cells independently getting mutations, so for $t \leq 1$ we have
$$P(\tau(M_1) \leq t) = 1-\left(\frac{v_2e^{-v_1 t}-v_1e^{-v_2 t}}{v_2-v_1}\right)^{N-1},$$
or equivalently,
$$P(\sqrt{v_1 v_2 N}\tau(M_1) \leq t) = 1-\left(\frac{v_2e^{-\sqrt{v_1/v_2 N} t}-v_1e^{-\sqrt{v_2/v_1 N} t}}{v_2-v_1}\right)^{N-1}.$$

By using the third degree Taylor expansion of the exponential function we get the bounds
$$1-\frac{t^2}{2N} - \sqrt{\frac{v_1^3}{v_2 N^3}}\frac{t^3}{6} \leq \frac{v_2e^{-\sqrt{v_1/v_2 N} t}-v_1e^{-\sqrt{v_2/v_1 N} t}}{v_2-v_1} \leq 1-\frac{t^2}{2N} + \sqrt{\frac{v_2^3}{v_1 N^3}}\frac{t^3}{6}.$$

Notice that $N \sqrt{v_1^3/v_2 N^3} = v_1^2/\sqrt{v_1 v_2 N} \rightarrow 0$ and $N \sqrt{v_2^3/v_1 N^3} = v_2^2/\sqrt{v_1 v_2 N} \rightarrow 0$. Then for any fixed $t$ we have
$$\left(1-\frac{t^2}{2N} - \sqrt{\frac{v_1^3}{v_2 N^3}}\frac{t^3}{6}\right)^{N-1} \rightarrow \frac{t^2}{2}$$
and
$$\left(1-\frac{t^2}{2N} + \sqrt{\frac{v_2^3}{v_1 N^3}}\frac{t^3}{6}\right)^{N-1} \rightarrow \frac{t^2}{2}.$$

If $v_1 = v_2$ then the probability that one cell has two mutations by time $t$ is $1-e^{-v_1 t} - v_1 t e^{-v_1 t}$ if we ignore splitting and apoptosis. The probability that one of the $N$ cells has two mutations by time $t$ is $1 - (e^{-v_1 t} - v_1 t e^{-v_1 t})^N$. By applying the same techniques as above we get $P(\sqrt{v_1 v_2 N} \tau(M_1) \leq t) \rightarrow 1-e^{-t^2/2}$ when $v_1 = v_2$.

Combining the two results above we have $P(\sqrt{v_1 v_2 N} \tau(M_1) \leq t) \rightarrow 1-e^{-t^2/2}$ when ignoring splitting and apoptosis. Then $P(\tau(M_1) < 1) = P(\sqrt{v_1 v_2 N}\tau(M_1) < \sqrt{v_1 v_2 N}) \rightarrow 1$. Therefore, the probability that two mutations occur before time 1 is converging to 1 so we may ignore splitting and apoptosis in this case. This gives the desired result for $\tau(M_1).$

By Corollary \ref{ddsigma2} we know that $\rho(M_1)$ converges in probability to 1. Because the two mutations occur before splitting or apoptosis, the probability that the cancer causing type-1 mutation and the first type-2 mutation are on the same cell converges to 1. Therefore, $\sigma(H_1)$ converges to 1 in probability.
\end{proof}

\section{The sd and ss regimes}
In this section we need two different models. The first one is the same as model $H_2$ except that only the stem cell receives type-1 mutations and only the daughter cells receive type-2 mutations. The second is the same as $H_2$ except that only the stem cell receives mutations. These will be referred to as models $M_2$ and $M_3$ respectively.

\begin{Prop} \label{ssRegime}
Let $X$ be a random variable which has an exponential distribution with parameter 1.
\begin{enumerate}
\item If $u_1 \ll 1/\log N$ and $u_1 \ll N v_2$ then $u_1 \tau(M_2) \rightarrow_d X$ and $\rho(M_2)$ converges in probability to $\alpha \wedge 1$.
\item If $u_1 \ll u_2$ then $u_1 \tau(M_3) \rightarrow_d X$.
\item Let $A \geq 1$ and $Z$ be an exponentially distributed random variable with parameter $1/A$ which is independent of $X$. If $u_1 \sim A u_2$ then $u_1 \tau(M_3) \rightarrow_d X + Z$.
\end{enumerate}
\end{Prop}
The goal of this section is to prove Lemma \ref{ssRegime}. It will be shown later that the conditions used in Lemma \ref{ssRegime} for the sd regime are the only relevant conditions.

Define $X_1$ to be the time at which the cancer causing type-1 mutation occurs and define $X_2$ to be the time after the cancer causing type-1 mutation until the first type-2 mutation. Note that because the stem cell is the only cell that gets type-1 mutations in models $M_2$ and $M_3$ that the first successful type-1 mutation is also the cancer causing type-1 mutation.

\begin{Lemma} \label{oddProp}
Consider the model $M_2$. For time $t \leq \log N$ after the stem cell receives a type-1 mutation we have
$$
P(X_2 > t) \geq e^{-2^{t+2}v_2} \mbox{ and }P(X_2 > t) \leq e^{-(2^{t-2}-2)v_2}.
$$
\end{Lemma}
\begin{proof}
First we establish the upper bound. After the stem cell gets the first mutation it takes at most one time unit until the mutation is passed along to the first generation daughter cell. Assuming it does take one time unit until the first generation daughter cell inherits the mutation we can get an upper bound on $P(X_2 > t)$. Let time $t=0$ denote the time at which the stem cell receives the type-1 mutation. There are no mutations being acquired by the daughter cells for time $t \in [0,1)$. For time $t \in [1,2)$ the generation 1 daughter cell is the only type-1 daughter cell. So for $t \in [1,2)$ we have $P(X_2 > t) = e^{-(t-1)v_2}$. For time $t \in [2,3)$ the first two generations have the mutation which is a total of 3 cells. Therefore, for $t \in [2,3)$ we have $P(X_2 > t) = e^{-(3(t-2)v_2 + v_2)}$ where the $v_2$ is added because of the probability of having a mutation before time 2. Extending this inductively gives us
$$P(X_2 > t) \leq e^{-[(2^{\lfloor t \rfloor}-1)(t - \lfloor t \rfloor)+\sum_{i = 2}^{\lfloor t \rfloor} (2^{i-1}-1)] v_2} \leq e^{(-2^{t-2}-1)v_2}$$
for any $t \leq \log N$.

For the lower bound we use the same reasoning as above except that we assume it takes 0 time for the generation 1 daughter cell to become a type-1 after the stem cell is a type-1. This gets us
$$P(X_2 > t) \geq e^{-[(2^{\lceil t \rceil}-1)(t - \lfloor t \rfloor)+\sum_{i = 1}^{\lfloor t \rfloor} (2^i-1)] v_2} \geq e^{-2^{t+2}v_2}.$$
\end{proof}

\begin{Lemma} \label{sdsigma}
The location of the second mutation satisfies $\rho(M_2) \rightarrow_p \alpha \wedge 1$.
\end{Lemma}
\begin{proof}
By Lemma \ref{oddProp} we have $P(X_2 > \log N) \geq e^{-4N v_2}.$ If $\alpha > 1$ then $P(X_2 > \log N)$ converges to 1 and the mutation will spread throughout the entire crypt. If this is the case then any cell is equally likely to have the second mutation. Therefore $P(\rho(M_2) \leq \beta) \leq (2^{\beta l}-1)/(2^l-1)$ for any $\beta \in [0,1)$ so $\rho(M_2) \rightarrow_p 1$.

Now suppose $\alpha \leq 1$. Let $\epsilon > 0$ so that $\alpha - \epsilon > 0$. Then by Lemma \ref{oddProp}
$$P(X_2 > l(\alpha - \epsilon)) \geq e^{-2^{l(\alpha - \epsilon) + 2}v_2}.$$
Because $4 N^{\alpha - \epsilon} v_2 \rightarrow 0$ we get the convergence $P(X_2 > l(\alpha - \epsilon)) \rightarrow 1$. By time $l(\alpha-\epsilon)$ the mutation will have spread to the first $\lfloor l(\alpha-\epsilon) \rfloor$ generations so that for times after $l(\alpha - \epsilon)$ we know that at least $2^{\lfloor l(\alpha-\epsilon)\rfloor}$ cells have the type-1 mutation. Therefore,
$$P(\{\rho(M_2) \leq \beta\} \cap \{X_2 > l(\alpha-\epsilon)\}) \leq (2^{\beta l}-1)/(2^{(\alpha-\epsilon)l-1}-1).$$ Thus, for any $\beta < \alpha - \epsilon$,
$$P(\rho(M_2) \leq \beta) < \frac{2^{\beta l}-1}{2^{(\alpha-\epsilon)l-1}-1} + P(X_2 \leq l(\alpha-\epsilon)) \rightarrow 0$$
Hence $P(\rho(M_2) \geq \alpha - \epsilon) \rightarrow 1$. Because $\epsilon$ may be arbitrarily small we have finished the case when $\alpha = 1$.

Suppose $\alpha < 1$ and let $\epsilon > 0$ so that $\alpha + \epsilon \leq 1$. Then by Lemma \ref{oddProp}
$$P(X_2 > l(\alpha + \epsilon)) \leq e^{-(2^{l(\alpha + \epsilon)-2} - 1) v_2}.$$
Because $N^{\alpha + \epsilon}v_2/4 \rightarrow \infty$, we have $P(X_2 > l(\alpha + \epsilon)) \rightarrow 0$. By time $l(\alpha + \epsilon)$ the mutation has only spread to the first $l(\alpha + \epsilon)$ generations, so $P(\rho(M_2) > \alpha + \epsilon) \rightarrow 0$ where $\epsilon$ is arbitrarily small.
\end{proof}

\begin{Lemma} \label{sdExp}
If $u_1 \ll 1/\log N$ and $u_1 \ll N v_2$ then $u_1 \tau(M_2) \rightarrow_d X$ where $X$ has exponential distribution with parameter 1.
\end{Lemma}
\begin{proof}
Since the stem cell is getting mutations according to a Poisson process at rate $u_1$ we have that $u_1 X_1$ is an exponentially distributed random variable with parameter 1. This leaves us to show $u_1 X_2 \rightarrow_p 0$.

Suppose we consider a new model $M_2'$ which is the same as model $M_2$ except that the type-2 mutations can only occur on daughter cells $\log N$ time after the stem cell has a type-1 mutation. We can couple models $M_2$ and $M_2'$ so that the same Poisson processes are marking the mutations on the daughter cells in each model but that any proposed type-2 mutation is rejected in model $M_2'$ until $\log N$ time after the stem cell mutation. This way $X_1$ is the same in models $M_2$ and $M_2'$. Also, if we let $X_2' = \tau(M_2')-X_1$ then $X_2' \geq X_2$. Therefore it is enough to show that $u_1 X_2' \rightarrow_p 0$.

If we wait $\log N$ time after the stem cell receives its type-2 mutation then all of the daughter cells will be type-1. Then the $(N-1)$ daughter cells are getting type-2 mutations at rate $v_2$. Thus for any fixed $N$ we have
$$P(X_2' > t) = 1_{[0,\log N]}(t) + e^{-v_2 (N-1) (t-\log N)}1_{(\log N,\infty]}(t).$$
Let $\epsilon > 0$. Then
$$P(u_1 X_2' > \epsilon) = 1_{[0,\log N]}\left(\frac{\epsilon}{u_1}\right) + e^{-v_2 (N-1) (\epsilon/u_1-\log N)}1_{(\log N,\infty]}\left(\frac{\epsilon}{u_1}\right).$$
By our assumptions, $u_1 \log N \rightarrow 0$ so for $N$ large enough this becomes
$$P(u_1 X_2' > \epsilon) = e^{-v_2 (N-1) (\epsilon/u_1-\log N)}.$$
Also by our assumptions, $-v_2 (N-1) (\epsilon/u_1-\log N) \sim -v_2 N \epsilon/u_1 \rightarrow -\infty$, so
$$P(u_1 X_2' > \epsilon) \rightarrow 0.$$
\end{proof}

\begin{proof}[Proof of Proposition \ref{ssRegime}]
Combining Lemmas \ref{sdsigma} and \ref{sdExp} we get part 1 of Proposition \ref{ssRegime}.

As in model $M_2$, $u_1 X_1$ has the exponential distribution with parameter 1. To prove part 2 of Proposition \ref{ssRegime} we need to show that $u_1 X_2 \rightarrow_p 0$. Let $\epsilon > 0$. Then
$$P(u_1 X_2 > \epsilon) = P(X_2 > \epsilon/u_1) = e^{-\epsilon u_2/u_1}.$$
Since $u_2/u_1 \rightarrow \infty$ we have $P(u_1 X > \epsilon) \rightarrow 0$.

Lastly we prove part 3 of Proposition \ref{ssRegime}. In model $M_3$ both mutations occur on the stem cell. In this case $u_1 X_1$ and $u_2 X_2$ are both exponentially distributed with parameter 1. Because $u_1 X_2 = (u_1/u_2) u_2 X_2$ we have that $u_1 X_2$ is exponentially distributed with parameter $u_2/u_1$. By assumption, $u_2/u_1 \rightarrow 1/A$ so $u_1 X_2$ converges in distribution to $Z$. The random variables $X_1$ and $X_2$ are independent for each $N$ so
$$u_1 \tau(M_3) = u_1 X_1+ u_1 X_2 \rightarrow_d X+Z.$$
\end{proof}

\section{Proof of the Theorem}
We will couple the models $H_2$, $M_1$, $M_2$ and $M_3$ so that the Poisson processes used in models $M_1$, $M_2$ and $M_3$ are the appropriate subcollections of Poisson processes which are used in model $H_2$. Let $T$ be the time that the stem cell becomes a type-1. Note, because the stem cell cannot inherit a type-1 mutation and $H_2$, $M_2$ and $M_3$ are coupled, that $T$ will be the same for models $H_2$, $M_2$ and $M_3$.

Let $X$ be exponentially distributed with parameter 1 and let $Y$ be a random variable with the Rayleigh distribution.

\begin{Lemma} \label{DR1}
Suppose $v_1 v_2 \ll 1/N(\log N)^2$. If $u_1 \ll v_1 v_2 N \log N$ then $P(\tau(M_1) < T) \rightarrow 1$. If $u_1 \gg v_1 v_2 N \log N$ then $P(\tau(M_3) < \tau(M_1)) \rightarrow 1$.
\end{Lemma}
\begin{proof}
By part 1 of Proposition \ref{ddRegime} $(\alpha \wedge 1) v_1 v_2 N (\log N) \tau(M_1) \rightarrow_d X$. Mutations to the stem cell occur at rate $u_1$ so $u_1 T \rightarrow_d X.$ Because the mutations Poisson processes which mark the mutations in model $M_1$ are independent of the Poisson process that marks the mutations on the stem cell, if $u_1 \ll v_1 v_2 N \log N$ then $P(\tau(M_1) < T) \rightarrow 1$ by Lemma \ref{ConLem}.

On the other hand, suppose $u_1 \gg v_1 v_2 N \log N$. We are assuming $u_1 \leq u_2$ so we could decrease $P(\tau(M_3) < \tau(M_1))$ by decreasing $u_2$ to $u_1$. Then the distribution of $u_1 \tau(M_3)$ is the distribution of the sum of two independent exponentially distributed random variables, $P(u_1 \tau(M_3) \leq t) \geq 1-e^{-t}-te^{-t}$. By Lemma \ref{ConLem}, $P(\tau(M_3) < \tau(M_1)) \rightarrow 1.$
\end{proof}

\begin{Lemma} \label{DR2}
Suppose $1/N (\log N)^2 \ll v_1 v_2 \ll 1/N$. If $u_1 \ll \sqrt{v_1 v_2 N}$ then $P(\tau(M_1) < T) \rightarrow 1$. If $u_1 \gg \sqrt{v_1 v_2 N}$ then $P(\tau(M_3) < \tau(M_1)) \rightarrow 1$.
\end{Lemma}
\begin{proof}
Let $u_1 \ll \sqrt{v_1 v_2 N}$. By part 2 of Proposition \ref{ddRegime} we have $\sqrt{v_1 v_2 N} \tau(M_1) \rightarrow_d Y$. The stem cell is getting mutations at rate $u_1$ so $u_1 T \rightarrow X$. The Poisson processes that are marking the mutations in model $M_1$ are independent of the Poisson process that marks mutations on the stem cell, so the result follows by Lemma \ref{ConLem}.

If $u_1 \gg v_1 v_2 N \log N$ then the proof follows by the same reasoning as used in Lemma \ref{DR1} when considering $u_1 \gg v_1 v_2 N \log N$.
\end{proof}

\begin{Lemma} \label{DR3}
If $v_1 v_2 \gg 1/N$ then $P(\tau(M_1) < T) \rightarrow 1$.
\end{Lemma}
\begin{proof}
By part 3 of Lemma \ref{ddRegime} we have $\sqrt{v_1 v_2 N} \tau(M_1) \rightarrow_d Y$. The stem cell is getting mutations at rate $u_1$ so $u_1 T \rightarrow X$. The Poisson processes that are marking the mutations in model $M_1$ are independent of the Poisson process that marks mutations on the stem cell, so the result follows by Lemma \ref{ConLem} since $u_1 \ll 1 \ll \sqrt{v_1 v_2 N}$.
\end{proof}

\begin{Lemma} \label{DR4}
If $u_2 \ll 1/\log N$ and $u_2 \ll N v_2$ then $P(\tau(M_2) < \tau(M_3)) \rightarrow 1$.
\end{Lemma}
\begin{proof}
Because models $M_2$ and $M_3$ are coupled, the stem cell in each model will receive a type-1 mutation at the same time. After this the Poisson processes marking the mutations in models $M_2$ and $M_3$ are independent. Let $T_2$ be the time it takes for a type-2 mutation to occur in model $M_2$ after the stem cell has a type-1 mutation and let $T_3$ be the time it takes for a type-2 mutation to occur in model $M_3$ after the stem cell has a type-1 mutation. Then $P(\tau(M_2) < \tau(M_3)) = P(T_2 < T_3)$.

Consider again the model $M_2'$ that was introduced in the proof of Lemma \ref{sdExp} which is the same as model $M_2$ except that the type-2 mutations can only occur on daughter cells $\log N$ time after the stem cell has a type-1 mutation. We can couple models $M_2$ and $M_2'$ as we did before so that the time of the stem cell mutation is the same in models $M_2$ and $M_2'$. Let $T_2'$ be the time it takes to acquire a type-2 mutation in model $M_2'$ after the stem cell has a type-1 mutation. Then $T_2' \geq T_2$ so it is enough to show that $P(T_2' < T_3) \rightarrow 1$.

If we wait $\log N$ time after the stem cell receives its type-2 mutation then all of the daughter cells will be type-1. Then the $(N-1)$ daughter cells are getting type-2 mutations at rate $v_2$. Thus for any fixed $N$ we have
$$P(T_2' > t) = 1_{[0,\log N]}(t) + e^{-v_2 (N-1) (t-\log N)}1_{(\log N,\infty]}(t).$$
Let $\epsilon > 0$. Then
$$P(T_2' < T_3) = P(T_2' < T_3|T_3 < \log N)P(T_3 < \log N) + P(T_2' < T_3|T_3 \geq \log N)P(T_3 \geq \log N).$$
Because $u_2 \ll 1/\log N$ and $u_2 T_3$ has the exponential distribution with parameter 1, $P(T_3 \geq \log N) \rightarrow 1$. The memoryless property of the exponential distribution gives us that
$$P(T_2' < T_3|T_3 \geq \log N) = \frac{v_2(N-1)}{v_2(N-1)+u_2} \rightarrow 1$$
which completes the proof.
\end{proof}

\begin{Lemma} \label{DR5}
If $u_2 \gg 1/\log N$ or $u_2 \gg N v_2$ then $P(\tau(M_3) < \tau(M_2)) \rightarrow 1$.
\end{Lemma}
\begin{proof}
Because models $M_2$ and $M_3$ are coupled, the stem cell in each model will receive a type-1 mutation at the same time. After this the Poisson processes marking the mutations in models $M_2$ and $M_3$ are independent. Let $T_2$ be the time it takes for a type-2 mutation to occur in model $M_2$ after the stem cell has a type-1 mutation and let $T_3$ be the time it takes for a type-2 mutation to occur in model $M_3$ after the stem cell has a type-1 mutation. Then $P(\tau(M_3) < \tau(M_2)) = P(T_3 < T_2)$.

Suppose $u_2 \gg 1/\log N$. By Lemma \ref{sdsigma} we know that $\rho(M_2) \rightarrow_p \alpha \wedge 1$. Therefore, if $0 < \delta < (\alpha \wedge 1)$ then $P(\rho(M_2) > (\alpha \wedge 1)-\delta) \rightarrow 1$. If $\rho(M_2) > (\alpha \wedge 1)-\delta$ then the second mutation occurs on a generation higher than $((\alpha \wedge 1)-\delta)l$. Since the stem cell is the only cell that gets type-1 mutations in model $M_2$, this means that $T_2 \geq \lfloor((\alpha \wedge 1)-\delta)l\rfloor$ because it takes that much time for the type-1 mutation to spread to the generation $((\alpha \wedge 1)-\delta)l$ daughter cells. On the other hand, in model $M_3$ the second mutation is occurring at rate $u_2$ so that $u_2 T_3$ is exponentially distributed with parameter 1. Then $P(T_3 < K \log N) = P(u_2 T_3 < u_2 K \log N) \rightarrow 1$ for any positive number $K$ since $u_2 \log N \rightarrow \infty$. Therefore $P(T_3 < T_2) \rightarrow 1$.

Suppose $u_2 \gg N v_2$. The rate at which type-2 mutations occur in model $M_2$ is always bounded by $(N-1) v_2$. Suppose we consider a new model $M_2''$ which is the same as $M_2$ except that once the stem cell has a type-1 mutation, all of the daughter cells also have a type-1 mutation instantaneously. This can be coupled so that after the stem cell gets a type-1 mutation then any type-2 mutation proposed by a Poisson process on a daughter cell is accepted in model $M_2''$. Then if we let $T_2''$ be the time it takes for a type-2 mutation to occur in model $M_2''$ after the stem cell has a type-1 mutation, $(N-1) v_2 T_2''$ has the exponential distribution with parameter 1. By Lemma \ref{ConLem}, $P(T_3 < T_2'') \rightarrow 1$. Because $T_2 \geq T_2''$ we have the desired result.
\end{proof}

\begin{proof}[Proof of Proposition \ref{mainTheorem}]
From the coupling we have $\tau(H_2) = \tau(M_1) \wedge \tau(M_2) \wedge \tau(M_3)$ because any type-2 mutation which  occurs in model $H_2$ must occur in at least one of the models $M_i$ for some $i$, and if a mutation occurs in model $M_i$ then it will also occur in model $H_2$.

Suppose $P(\tau(M_1) < T) \rightarrow 1$. Before time $T$ models models $M_2$ and $M_3$ are only acquiring mutation on the stem cell. Therefore, models $M_2$ and $M_3$ only have type-0 cells before time $T$ and $P(\tau(M_1) < \tau(M_2) \wedge \tau(M_3)) \rightarrow 1$.

\begin{itemize}
\item By Lemma \ref{DR1} if $v_1 v_2 \ll 1/N(\log N)^2$ and $u_1 \ll v_1 v_2 N \log N$ then $P(\tau(M_1) < T) \rightarrow 1$ so by part 1 of Proposition \ref{ddRegime} and the coupling of $H_2$ with $M_1$ we have $(\alpha \wedge 1) v_1 v_2 N (\log N) \tau(H_2) \rightarrow_d X$. The distribution of $\sigma(H_2)$ converges to a uniform distribution on $((1-\alpha)^+, 1]$ and $\rho(H_2)$ converges in distribution to 1.
\item By Lemma \ref{DR2} if $1/N (\log N)^2 \ll v_1 v_2 \ll 1/N$ and $u_1 \ll \sqrt{v_1 v_2 N}$ then $P(\tau(M_1) < T) \rightarrow 1$ so by part 2 of Proposition \ref{ddRegime} and the coupling of $H_2$ with $M_2$ we have $\sqrt{v_1 v_2 N} \tau(H_2) \rightarrow_d Y$. Both $\sigma(H_2)$ and $\rho(H_2)$ converge in distribution to 1.
\item By Lemma \ref{DR3} if $v_1 v_2 \gg 1/N$ then $P(\tau(M_1) < T) \rightarrow 1$ so by part 3 of Proposition \ref{ddRegime} and the coupling of $H_2$ with $M_1$ we have $\sqrt{N v_1 v_2} \tau(H_2) \rightarrow_d Y$. Both $\sigma(H_2)$ and $\rho(H_2)$ converge in distribution to 1.
\end{itemize}

If either $v_1 v_2 \ll 1/N(\log N)^2$ and $u_1 \gg v_1 v_2 N \log N$ or $1/N(\log N)^2 \ll v_1 v_2 \ll 1/N$ and $u_1 \gg \sqrt{v_1 v_2 N}$ then $P(\tau(M_3) < \tau(M_1)) \rightarrow 1$ by Lemmas \ref{DR1} and \ref{DR2} respectively. Therefore, $P(\tau(M_2) \wedge \tau(M_3) < \tau(M_1)) \rightarrow 1$ (meaning that the cancer causing type-1 mutation occurs on the stem cell). Given these four conditions, we are left only to compare $\tau(M_2)$ and $\tau(M_3)$.

\begin{itemize}
\item By Lemma \ref{DR4} if $u_2 \ll 1/\log N$ and $u_2 \ll N v_2$ then $P(\tau(M_2) < \tau(M_3)) \rightarrow 1$. Because $u_1 \leq u_2$ the hypotheses are true for $u_1$ as well. Therefore, by the coupling of $H_2$ with $M_2$ and part 1 of Proposition \ref{ssRegime} we have $u_1 \tau(H_2) \rightarrow_d X$. The distribution of $\rho(H_2)$ converges to $\alpha \wedge 1$.

\item By Lemma \ref{DR5} if $u_2 \gg 1/\log N$ or $u_2 \gg N v_2$ then $P(\tau(M_3) < \tau(M_2)) \rightarrow 1$. If $u_1 \ll u_2$ then by the coupling of $H_2$ with $M_3$ and part 2 of Proposition \ref{ssRegime} we have $u_1 \tau(H_2) \rightarrow_d X$. If $u_1 \sim A u_2$ then by the coupling of $H_2$ with $M_3$ and part 3 of Proposition \ref{ssRegime} we have $u_1 \tau(H_2) \rightarrow_d X + Z$ where $Z$ is an exponentially distributed random variable with parameter $1/A$ that is independent of $X$.
\end{itemize}
By Lemma \ref{coupleLem} the results hold for model $H_1$ as well.
\end{proof}

\section{The Null Model}
For this section we always have $u_1 = u_2 = v_1 = v_2 = \mu$ and we prove Proposition \ref{nullProp} for model $H_2$. Then Proposition \ref{nullProp} will hold for model $H_1$ as well by Lemma \ref{coupleLem}. We begin this section by pointing out that the conditions of part 5 of Theorem \ref{mainTheorem} always fail in the null model. The two conditions in the first conjunction become $\mu \ll 1/N \log N$. Of the two conditions in the second conjunction, one becomes $\sqrt{N} \ll 1$ which always fails. This reduces all of the conditions in the first bullet point to $\mu \ll 1/N \log N$. The conditions in the second bullet point become $\mu \gg 1/\log N$ or $1 \gg N$, so the conditions in part 5 can only hold if $1/\log N \ll \mu \ll 1/N \log N$ which can never happen.

This shows that the probability that the first type-2 mutation occurs on the stem cell converges to 0. For this reason, we will never consider model $M_3$ in this section.

\begin{proof}[Proof of part 2 of Proposition \ref{nullProp}]
This time we first consider independent models $M_1$ and $M_2$ meaning we do not couple the Poisson processes that mark mutations on the cells within each model. We construct a new model from models $M_1$ and $M_2$ which we will refer to as model $M_{1,2}^-$. The Poisson processes of model $M_2$ always mark the cells in model $M_{1,2}^-$. The Poisson processes that mark the mutations in model $M_1$ mark the daughter cells in model $M_{1,2}^-$ until the stem cell has a type-1 mutation. After the stem cell has a type-1 mutation, the Poisson processes in model $M_1$ no longer mark any of the cells in model $M_{1,2}^-$. This way model $M_{1,2}^-$ behaves exactly like model $M_2$ after the stem cell has a type-1 mutation.

Model $M_{1,2}^-$ is the same as model $H_2$ except that the stem cell cannot get type-2 mutations and for $\log N$ time after the stem cell receives a mutation the type-2 mutations are suppressed on daughter cells that have not inherited the type-1 mutation from the stem cell. Let $T$ be the time at which the stem cell has a mutation and let $T_2 = \tau(M_2) - T$. By the same argument used in Lemma \ref{sdExp} to show $u_1 X_2' \rightarrow_p 0$ we know $\mu T_2 \rightarrow_p 0$. Also, from part 1 of Proposition \ref{ddRegime} we know $P(\mu \tau(M_1) > t) \rightarrow e^{-At}$. Let $\epsilon > 0$. Then
\begin{align*}
& \limsup P(\mu \tau(H_2) > t) \leq \limsup P(\mu \tau(M_{1,2}^-) > t) \\
& = \limsup \left(P(\{\mu \tau(M_1) > t\} \cap \{\mu T > t\}) + P(\{\tau(M_1) > T\} \cap \{\mu T \leq t\} \cap \{\mu \tau(M_2) > t\})\right) \\
& \leq \limsup P(\{\mu \tau(M_1) > t\})P(\{\mu T > t\}) + \limsup P(\{\mu T \leq t\} \cap \{\mu \tau(M_2) > t\}) \\
& = e^{-(1+A)t}+\limsup P(\{\mu T \leq t\} \cap \{\mu(T+T_2) > t\}) \\
& \leq e^{-(1+A)t} \\
& \hspace{20pt} + \limsup (P(\{\mu T \leq t \leq \mu(T+\epsilon)\} \cap \{\mu T_2 \leq \epsilon\}) + P(\{\mu T \leq t \leq \mu(T+T_2)\} \cap \{\mu T_2 > \epsilon\})) \\
& \leq e^{-(1+A)t}+\limsup P(\mu T \in [t-\epsilon\mu,t]) + \limsup P(\mu T_2 > \epsilon) \\
& = e^{-(1+A)t}
\end{align*}
where the third line follows by the independence of $\tau(M_1)$ and $T$ and the last line follows because $\mu T$ is exponentially distributed, $\mu \rightarrow 0$ and $\mu T_2 \rightarrow_p 0$. Hence we have $\limsup P(\mu \tau(H_2) > t) \leq e^{-(1+A)t}$.

We define another model, $M_{1,2}^+$, which is the same as model $M_{1,2}^-$ except that we always count the mutations from model $M_1$. That is, we have models $M_1$ and $M_2$ and we are looking for the first mutation that occurs on either of these models. We couple model $M_{1,2}^+$ with model $H_2$ so that before the stem cell has a mutation the Poisson processes marking the mutations in models $M_1$, $M_2$ and $H_2$ are the same and after the stem cell has a mutation the Poisson processes marking the mutations in model $M_2$ only mark those generations in model $H_2$ which have not yet inherited a type-1 mutation from the stem cell. If we wait $\log N$ time after the stem cell has a mutation then the Poisson processes marking model $M_1$ are not marking model $H_2$. Then
\nonumber
\begin{align}
\liminf P(\mu \tau(H_2) > t) & \geq \liminf P(\mu M_{1,2}^+ > t) \\
& = \liminf P(\{\mu \tau(M_1) > t\} \cap \{\mu \tau(M_2) > t\}) \\
& = \liminf P(\mu \tau(M_1) > t)P(\mu \tau(M_2) > t) \\
& = e^{-(1+A)t}
\end{align}
where the last equality follows by part 1 of Proposition \ref{ddRegime}. Combining this with the above result we have $\lim P(\mu \tau(H_2) < t) = 1 - e^{-(1+A)t}$.

Note that $\limsup P(T = \tau(M_1)) = 0.$ Let $\epsilon > 0$. By continuity of measure there exists $\delta > 0$ such that $\limsup P(0 \leq \tau(M_1)-T \leq \delta) < \epsilon.$ Then
\nonumber
\begin{align}
& \limsup P(\tau(M_1) < \tau(M_2)) \\
& = \limsup (P(\{\tau(M_1) < \tau(M_2)\} \cap \{\tau(M_1) < T\}) + P(\{\tau(M_1) < \tau(M_2)\} \cap \{\tau(M_1) \geq T\})) \\
& \leq \limsup P(\tau(M_1) < T) + \limsup P(T \leq \tau(M_1) \leq T + T_2) \\
& \leq \frac{A}{1+A} + \limsup P(T \leq \tau(M_1) \leq T+\delta | \mu T_2 \leq \delta) \\
& \leq \frac{A}{1+A} + \epsilon
\end{align}
where the fourth line follows because $\mu T_2 \rightarrow_p 0$ and by Lemma \ref{CondLim}. Because $\epsilon > 0$ was arbitrary we have $\limsup P(\tau(M_1) < \tau(M_2)) \leq A/(1+A)$. On the other hand,
$$\liminf P(\tau(M_1) < \tau(M_2)) \geq \liminf P(\{\tau(M_1) < T\}) = \frac{A}{1+A}.$$

Let $Z$ be a random variable such that $Z = 1$ if $\tau(M_1) < \tau(M_2)$ and $Z = 0$ otherwise. If $\tau(M_1) < \tau(M_2)$ then $\sigma(H_2) = \sigma(M_1)$ and $\rho(H_2) = \rho(M_1)$. By Proposition \ref{ddRegime} we know that $\sigma(M_1)$ converges in distribution $U$ and $\rho(M_1)$ converges in probability to 1. The event $\tau(M_1) = \tau(M_2)$ has probability 0. If $\tau(M_1) > \tau(M_2)$ then $\sigma(H_2) = \sigma(M_2)$ and $\rho(H_2) = \rho(M_2)$. By definition of model $M_2$ we always have $\sigma(M_2) = 0$ and by Proposition \ref{ssRegime} $\rho(M_2)$ converges in probability to $\alpha \wedge 1$. Therefore,
$$\sigma(H_2) = \sigma(M_1)Z + \sigma(M_2)(1-Z) \rightarrow_d U\xi$$
and
$$\rho(H_2) = \rho(M_1)Z + \rho(M_2)(1-Z) \rightarrow_d \xi+(\alpha \wedge 1)(1-\xi).$$
\end{proof}

Let $\mathcal{N}$ be the set of Radon measures $\nu$ on a Polish space $(\Psi, \mathcal{B})$ where $\mathcal{B}$ is the Borel $\sigma$-field such that $\nu(\{x\}) \in \N \cup \{0,\infty\}$ for all $x \in \Psi$. For the next proof we will consider a point process to be a random variable taking on elements of $\mathcal{N}$. We consider $\nu(\{x\})$ to be the number of times the point $x$ has been marked. For a Poisson point process whose intensity measure has no atoms $\nu(\{x\})$ is 0 or 1 for all $x$ and $\{x \in \Psi: \nu(\{x\}) > 0\}$ is discrete with probability 1.

Let $\Psi = [0,\infty) \times [0,1]$. The Poisson point process of successful type-1 mutations in model $M_1$ induces a point process on $\Psi$ where if a successful type-1 mutation occurs at time $t$ on a cell in generation $i$ in model $M_1$ then there is a point of $\Psi$ at $(t/l,i/l)$. We will call this point process $P_M$.

\begin{Lemma} \label{PPLem}
The limiting distribution of $P_M$ is a Poisson point process $P_\infty$ which has intensity measure $\nu' = A^2(\lambda \times \lambda_{[1/2,1]})$ where $\lambda$ is Lebesgue measure and $\lambda_{[1/2,1]}$ is the measure defined by $\lambda_{[1/2,1]}(B) = \lambda(B \cap [1/2,1])$ for any Lebesgue measurable set $B$.
\end{Lemma}
\begin{proof}
We let $C_C(\Psi,[-1,0])$ be the set of continuous functions $h:\Psi \rightarrow [-1,0]$ such that the set $\{\psi \in \Psi: h(\psi) \neq 0\}$ is precompact. Recall that a point process $X$ has an associated generating functional $\mathfrak{F}: C_C(\Psi,[-1,0]) \rightarrow \R$ defined by
$$\mathfrak{F}(h) = E[\prod_{\psi \in \Psi} (h(\psi)+1)^{\nu(\psi)}]$$
where $\nu$ is a Radon measure on $\Psi$ as described above. Probability generating functionals uniquely determine the distribution of point processes (see Theorem 14 of section 29.5 in \cite{FG}). Moreover, a sequence of point processes converges in distribution to a point process if and only if the corresponding sequence of generating functionals converges pointwise to a functional $\mathfrak{F}$ that satisfies the following: If $h_m$ is in the domain of $\mathfrak{F}$ for each $m$, $\bigcup_{m=1}^{\infty}\{\psi: h_m(\psi) \neq 0\}$ is relatively compact, and $h_m(\psi) \rightarrow 0$ as $m \rightarrow \infty$ for each $\psi$, then $\mathfrak{F}(h_m) \rightarrow 1$ as $m \rightarrow \infty$. In this case $\mathfrak{F}$ is the probability generating functional of the limiting point process (see Theorem 20 of Section 29.7 in \cite{FG}).

Notice that for any $N$ the points marked in $\Psi$ will all have coordinates $(x,y)$ where $y$ takes values in $\{1/\log N,2/\log N, \dots, 1\}$. We know from the proof of Corollary \ref{longCor} that the rate at which mutations occur along generation $i$ is bounded between $2^{i-1}\mu(1-e^{-\mu(2^{l-i+1}-2)})$ and $2^{i-1}\mu(1-e^{-\mu(2^{l-i+1}-1)})$. Therefore, if we look at the points that are marked in $\Psi$ whose second coordinate is fixed at $i/\log N$, the rate at which the marking will occur will be between $(\log N) 2^{i-1}\mu(1-e^{-\mu(2^{l-i+1}-2)})$ and $(\log N) 2^{i-1}\mu(1-e^{-\mu(2^{l-i+1}-1)})$ where the $\log N$ appears because time is scaled by $1/\log N$. This observation will allow us to work with time homogeneous Poisson point processes.

Let $\mathfrak{F}$ denote the generating functional associated with $P_M$. Let $\mathfrak{F}_1$ be the generating functional associated with the Poisson process on $\Psi$ which marks points at rate $(\log N)2^{i-1}\mu(1-e^{-\mu(2^{l-i+1}-2)})$ on each generation and let $\mathfrak{F}_2$ be the generating functional associated with the Poisson process on $\Psi$ which marks points at rate $(\log N)2^{i-1}\mu(1-e^{-\mu(2^{l-i+1}-1)})$ along each generation. Call the time homogeneous Poisson point processes $P_1$ and $P_2$ respectively. Because the intensity measure of $P_M$ is always between the intensity measures of $P_1$ and $P_2$ we have the bounds $\mathfrak{F}_1 \leq \mathfrak{F} \leq \mathfrak{F}_2$.

Let $X$ be a Poisson process with intensity measure $\nu$. It is known that the probability generating functional associated with $X$ is
$$\mathfrak{P}(h) = e^{-\int_{\Psi}h d\nu}.$$
To show a sequence of Poisson processes $\{X_n\}_{n=0}^\infty$ with intensity measures $\{\nu_n\}_{n=0}^\infty$ converges in distribution to a Poisson process $X$ with intensity measure $\nu$ it is enough to show that $\{\nu_n\}_{n=0}^\infty$ converges weakly to $\nu$. That is, for each $h \in C_C(\Psi,[-1,0])$ we need $\int_{\Psi} h d\nu_n \rightarrow \int_{\Psi} h d\nu$ as $n \rightarrow \infty$. Let $\nu_N^1$ be the intensity measure of $P_1$ when there are $N$ cells in the population and let $\nu_N^2$ be the intensity measure of $P_2$ when there are $N$ cells in the population. The goal is to show $\nu_N^1$ and $\nu_N^2$ both converge weakly to $\nu'$. Then the limiting distribution of $P_M$ will be $P_\infty$.

Let $R = (a,b] \times (c,d] \subset \Psi$. Then
$$\nu_N^1 (R) = (a-b)(\log N)\sum_{i \in (lc,ld]}2^i \mu (1-e^{-\mu(2^{l-i+1}-2)}) \rightarrow A^2 (d-c \vee \frac{1}{2})^+(b-a) = \nu'(R)$$
by Lemma \ref{longlem} and the assumption that $\mu \sim A/ \sqrt{N} \log N$ which implies $\mu^2 N \log N \sim A^2/\log N$. Now let $O$ be any open subset of $\Psi$. We can write $O = \bigcup_{n=1}^\infty R_n$ where each $R_n$ is a half open rectangle in the same form as $R$ above and the sets $\{R_n\}_{n=1}^\infty$ are pairwise disjoint. Then
$$\liminf_{N \rightarrow \infty}\nu_N^1(O) = \liminf_{N \rightarrow \infty} \sum_{j=1}^\infty \nu_N^1 (R_j) \geq \sum_{j=1}^\infty \nu'(R_j) = \nu'(O)$$
where the inequality follows by Fatou's lemma. By the same reasoning $\liminf \nu_N^2(O) \geq \nu'(O)$ for any open subset $O$ of $\Psi$. It follows by the Portmanteau Theorem that both $\nu_N^1$ and $\nu_N^2$ converge weakly to $\nu'$ as $N$ goes to infinity. Because of the bounds on the linear functionals we have that the limiting distribution of $P_M$ is a Poisson process with intensity $\nu'$.
\end{proof}

The notation used in Lemma \ref{PPLem} will also be used in this proof.

\begin{proof}[Proof of part 4 of Proposition \ref{nullProp}]
Notice that this is the boundary between two cases that are determined by model $M_1$. By Corollary $\ref{ddsigma2}$ we know $\rho(M_1) \rightarrow_p 1$ for all conditions that we are considering. Therefore, $\rho(H_1) \rightarrow_p 1$ in this case.

The strategy is to define functions $g$ and $h$ on the set of Radon measures that are continuous everywhere except a set of measure 0. Then we will apply the Continuous Mapping Theorem to get the desired convergence in distribution. Let $D$ be the subset of $\mathcal{N}$ such that $\nu \in D$ if there exists $(x,y) \in \Psi$ and $t \in \R$ such that $\nu(x,y) > 0$ and $\nu(x+t,y+t) > 0$. For all $t \geq 0$ define sets $T_t = \{(x,y): 0 \leq y \leq 1/2 \mbox{ and } 0 \leq x \leq y+t-1 \} \subset \Psi.$ These sets correspond the the triangles and quadrilaterals that were shown in the picture in the introduction. Let $V = \{(x,y) \in \Psi: \nu(x,y) > 0\}$ and define $t_0 = \inf \{t:V \cap T_t \neq \varnothing\}$. Define
$$g(\nu) = \lim_{\epsilon \rightarrow 0} \sup_y \{y:(x,y) \in V \cap T_{t_0+\epsilon} \mbox{ for some } x\}$$
and $h(\nu) = t_0$.

Given a Poisson point process $P$ on $\Psi$ whose intensity has no atoms, we can project the points of $P$ onto the line $y = -x$ in $\R^2$ along perpendicular angles of $\pi/4$. With probability 1 no two points of $P$ will be mapped to the same point under the projection. That is, under the law of $P$, $D$ has probability 0. Moreover, with probability 1 there will be no limit points under the projection. Therefore, under the intensity measure $A^2 (\lambda_{[1/2,1]} \times \lambda)$, there exists a unique point $(x_0,y_0) \in V \cap T_{t_0}$ and an $\epsilon > 0$ such that $V \cap T_{t_0+\epsilon} = \{(x_0,y_0)\}$ with probability 1. By definition $g(P) = y_0$. We claim that $g$ and $h$ are continuous at any Radon measure $\nu \in \mathcal{N} \backslash D$.

Let $\nu \in \mathcal{N} \backslash D$ and let $\{\nu_n\}_{n=1}^\infty$ be a sequence of Radon measures that converges weakly to $\nu$. Let $\epsilon > 0$ and  let $(x_0,y_0)$ be the unique point of $T_{t_0+\epsilon}$ such that $\nu(x_0,y_0) > 0$. For each point $(x',y') \in \Psi$ and every natural number $m$ define a function
$$f_{(x',y'),m}(x,y) = \left\{ \begin{array}{ll}
-1 & \mbox{ if } |(x,y)-(x',y')| < \epsilon/m\\
-(2-m|(x,y)-(x',y')|/\epsilon) & \mbox { if } \epsilon/m \leq |(x,y)-(x',y')| \leq 2\epsilon/m \\
0 & \mbox{ otherwise}
\end{array}
\right.$$
For $m$ large enough we have $\int_\Psi f_{(x_0,y_0),m}(x,y) d\nu = -1$ so $\int_\Psi f_{(x_0,y_0),m}(x,y) d\nu_n \rightarrow -1$ as $n \rightarrow \infty$ for large enough values of $m$. Because we can make $m$ arbitrarily large, there must be a sequence of points $\{(x_n,y_n)\}_{n=1}^\infty$ such that $\nu_n(x_n,y_n) = -1$ for all $n$ and $(x_n,y_n) \rightarrow (x_0,y_0)$ as $n \rightarrow \infty$. Likewise, for any point $(x',y') \in T_{t_0+\epsilon}$ there exists a large enough $m$ such that $\int_\Psi f_{(x',y'),m}(x,y) d\nu = 0$ so $\int_\Psi f_{(x',y'),m}(x,y) d\nu_n \rightarrow 0$ as $n \rightarrow \infty$. This shows that for $n$ large enough the Radon measures $\nu_n$ will assign measure 0 to all points in a ball of radius $\epsilon/m$ about $(x',y')$. From this it is easy to conclude $g(\nu_n) \rightarrow g(\nu)$ and $h(\nu_n) \rightarrow h(\nu)$. Therefore, $g$ and $h$ are both continuous on $\mathcal{N} \backslash D$. By Lemma \ref{PPLem} and the Continuous Mapping Theorem $g(P_M)$ converges in distribution to $g(P_\infty)$ and $h(P_M)$ converges in distribution to $h(P_\infty)$.

The next goal is to show that $g(P_M) - \sigma(M_1) \rightarrow_p 0$ and $h(P_M) - \tau(M_1)/\log N \rightarrow_p 0$. Then we will have that $\sigma(M_1) \rightarrow_d g(P_\infty)$ and $\tau(M_1)/\log N \rightarrow_d h(P_\infty)$. To achieve this we will first show that the probability that $(x_0, y_0)$ corresponds to the cancer causing type-1 mutation converges in probability to 1. Suppose $(x_0,y_0)$ does not correspond to the cancer causing type-1 mutation and let $(x_1,y_1)$ denote the point in $\Psi$ corresponding to the cancer causing type-1 mutation in $M_1$. Let $\epsilon > 0$ and suppose that $(x_1,y_1) \notin T_{t_0+\epsilon}$. The point $(x_0,y_0) \in T_{t_0}$ corresponds to a successful type-1 mutation in model $M_1$, and by the way that model $M_1$ marks points in $\Psi$ there will be a type-2 mutation in model $M_1$ that corresponds to a point in $T_{t_0}$. The ray starting at $(x_1,y_1)$ with an angle of $\pi/4$ will represent all of the descendants of the cancer causing type-1 mutation. The point on this line whose first coordinate is $t_0$ will be $(t_0,y'')$ where $y'' \leq 1-\epsilon$. In this case $\rho(M_1) = y'' \leq 1-\epsilon$. Let $E_1$ be the event that $(x_0,y_0)$ is the point in $\Psi$ that corresponds to the cancer causing type-1 mutation and $E_2$ be the event that two or more points occur in $T_{t_0+\epsilon}$. We know that $P_M$ converges in distribution to $P_\infty$ by Lemma \ref{PPLem} so
\begin{align*}
\limsup P(E_1^C) & = \limsup (P(E_1^C \cap \{(x_1,y_1) \in T_{t_0+\epsilon}\})+P(E_1^C \cap \{(x_1,y_1) \notin T_{t_0+\epsilon}\})) \\
& \leq \limsup P(E_2)+ \limsup P(\rho(M_1) < 1-\epsilon) \\
& \leq \frac{A^2}{2}\epsilon
\end{align*}
where the last line follows because $\rho(M_1) \rightarrow_p 0$ and $P(E_2) \leq P(V \cap (T_{t_0+\epsilon}\backslash T_{t_0}) \neq \varnothing)$. Because $\epsilon > 0$ was chosen arbitrarily, we have $\lim P(E_1^C) = 0$.

The above has established that $\lim P(E_1) = 1$. By definition of $\sigma(M_1)$ and $g(P_M)$ it is clear that
$$P(\sigma(M_1) - g(P_M) = 0|E_1) = 1$$
because $\sigma(M_1) = g(P_M) = y_0.$ Conditional on the event $E_1$ we also know that the cancer causing type-1 mutation occurs at time $(\log N) x_0$. Let $(x_0',y_0')$ be the point in $\Psi$ that corresponds to the type-2 mutation in $M_1$, so that $\rho(M_1) = y_0'$. Let $\nu$ be the Radon measure of points in $\Psi$ induced by $M_1$ and consider the fact that the descendants of the cancer causing type-1 mutation will lie on a line starting at $(x_0,y_0)$ with angle $\pi/4$. It is clear that $h(\nu) = t_0 = x_0 + 1 - y_0$ and $\rho(M_1) = y_0+\tau(M_1)/\log(N)-x_0$. Thus, if $h(\nu) - \tau(M_1)/\log N > \epsilon$ then $1-\rho(M_1) > \epsilon$, or equivalently $\rho(M_1) < 1-\epsilon$. Therefore, because $P(E_1) \rightarrow 1$,
$$P(h(P_M)-\tau(M_1)/\log N > \epsilon|E_1) = P(\rho(M_1) < 1-\epsilon|E_1) \rightarrow 0.$$
Again using the fact that $P(E_1) \rightarrow 1$ we get the desired result.

Now we are left to show that $g(P_\infty)$ and $h(P_\infty)$ have the distributions that are stated in part 4 of Proposition \ref{nullProp}. We have $P(h(P_\infty) \leq t)$ is the probability that a point of the Poisson process with intensity $A^2 (\lambda_{[1/2,1]} \times \lambda)$ has been marked in $T_t$. For $t \leq 1/2$ this is $1-e^{-A^2 t^2/2}$ and for $t > 1/2$ this is $1-e^{-A^2 t/2+A^2/8}$. Therefore,
$$P(\tau(M_1)/\log N \leq t) \rightarrow (1-e^{-A^2 t^2/2})1_{[0,1/2]}(t)+(1-e^{-A^2 t/2+A^2/8})1_{(1/2, \infty)}(t).$$

To find the distribution of $g(P_\infty)$ we will use the joint density function of $g(P_\infty)$ and $h(P_\infty)$. From the above computation it is clear that the density of $h(P_\infty)$ is
$$f_h(t) = A^2 te^{-A^2 t^2/2}1_{[0,1/2]}(t)+\frac{A^2}{2}e^{-A^2 t/2+A^2/8}1_{(1/2,\infty)}(t).$$
Conditioned on the event that $h(P_\infty) = t$ we know that $g(P_\infty)$ will have uniform distribution. If $t \leq 1/2$ then $g(P_\infty)$ is uniformly distributed on the interval $[1-t,1]$. If $t > 1/2$ then $g(P_\infty)$ is uniformly distributed on $[1/2,1]$. This gives us the conditional density function
$$f_{g|h}(s|t) = \left\{ \begin{array}{ll}
\frac{1}{t} & \mbox{ if } 1-t \leq s \leq 1 \mbox{ and } 0 \leq t \leq \frac{1}{2} \\
2 & \mbox{ if } \frac{1}{2} \leq s \leq 1 \mbox{ and } t > \frac{1}{2} \\
\end{array}
\right. .$$
Therefore, the joint density function of $g(P_\infty)$ and $h(P_\infty)$ is
$$f(s,t) = A^2e^{-A^2t^2/2}1_{[0,1/2]}(t)1_{[1-t,1/2]}(s)+A^2e^{-A^2t/2+A^2/8}1_{(1/2,\infty)}(t)1_{[1/2,1]}(s).$$
Integrating over $t$ we find that the density of $g(P_\infty)$ is
$$f_g(s) = \left(\int_{1-s}^{1/2}A^2e^{-A^2 t^2/2}dt+2e^{-A^2/8}\right)1_{[1/2,1]}(s).$$

This gives the desired limiting distribution for model $M_1$. By the same arguments as used above, the results will hold for model $H_1$ as well.
\end{proof}

\begin{proof}[Proof of part 6 of Proposition \ref{nullProp}]
First we change model $M_1$ so that only generation $l-1$ will get type-1 mutations and generation $l$ will get type-2 mutations. Also, assume that only one of the daughters will keep a mutation when the cells split so that if a type-1 cell splits it has a type-0 daughter and a type-1 daughter. The rate at which the type-1 mutations occur will be $\mu N/4$ since there are $N/4$ cells in generation $l-1$. Note that $\mu N/4 \sim A\sqrt{N}/4$. The probability that a type-1 mutation will have a type-2 descendant is $1-e^{\mu t} \sim \mu t \sim A t / \sqrt{N}$. Therefore, the type-2 mutations occur according to a Poisson process whose intensity measure $\nu$ satisfies $\nu([0,t]) \geq (A\sqrt{N}/4)(A t/\sqrt{N}) = A^2 t/4$. We have may have to wait up to two time units for the type-2 mutation to occur after the successful type-1 appears. For the sake of a lower bound we will always assume it takes 2 time units after a successful type-1 mutation until the type-2 mutation. By coupling this model with model $H_2$, this gets us $\liminf P(\tau(M_1) \leq t) \geq 1-e^{-2-A^2 t/4}$.

For the upper bound we change model $M_1$ so that the type-1 cells never undergo apoptosis. There are $N$ cells getting type-1 mutations so the type-1 mutations occur at rate $\mu N \sim A \sqrt{N}.$ If we wait $t$ time until after a type-1 mutation has occurred the cell will have at most $2^t$ descendants. If the type-1 mutation had occurred at time 0 and all of the descendants had existed since the type-1 mutation occurred then the probability that one of the cells had acquired a type-2 mutation would be $t 2^{\lfloor t \rfloor} \mu \leq t 2^t \mu \sim t 2^t A/\sqrt{N}$. Because the type-1 mutation may occur after time 0 and there have not been $2^t$ descendants with the type-1 mutation since the mutation occurred this is an upper bound on the probability that a type-2 mutation has occurred by time $t$. Therefore, the type-2 mutations occur according to a Poisson process with intensity $\nu([0,t]) \leq (A \sqrt{N})(t 2^t A/\sqrt{N}) = t 2^t A^2.$ Then $\limsup P(\tau(M_1) \leq t) \leq 1-e^{-A^2 2^t t}$. This shows part 6 of Proposition \ref{nullProp} with $c = 1-e^{-2-A^2 t/4}$ and $C = 1-e^{-A^2 2^t t}$.

By Corollary \ref{ddsigma2} we know $\rho(M_1) \rightarrow 1$. By the definitions of $\sigma(M_1)$ and $\rho(M_1)$ for any $\epsilon > 0$ if $\rho(M_1) - \sigma(M_1) > \epsilon$ then $\tau(M_1) > \epsilon \log N$. Therefore,
$$P(\rho(M_1)-\sigma(M_1) > \epsilon) \leq P(\tau(M_1) > \epsilon \log N) \leq e^{-A^2 2^{\delta \log N}(\delta \log N)} \rightarrow 0.$$
Let $\epsilon > 0$ and $\delta > 0$ and choose $N$ large enough so that $P(1-\rho(M_1) > \epsilon/2) < \delta/2$ and $P(\rho(M_1)-\sigma(M_2) > \epsilon/2) < \delta/2$. Then
\begin{align*}
P(1-\sigma(M_1) > \epsilon) & = P(1-\rho(M_1)+\rho(M_1)-\sigma(M_1) > \epsilon) \\
& \leq P(1-\rho(M_1) > \epsilon/2) + P(\rho(M_1)-\sigma(M_1) > \epsilon/2) \\
& < \delta.
\end{align*}
Therefore, $\sigma(M_1) \rightarrow_p 1$.

Using the same techniques as in the previous sections, we get the same results for $H_1$.
\end{proof}

\section{Acknowledgements}
I would like to thank Jason Schweinsberg for patiently helping me work through various parts of the problem and for helping to revise the first drafts of the paper.

\end{document}